\documentclass[11 pt, reqno]{amsart}
\usepackage{amsmath}
\usepackage{amsthm}
\usepackage[english]{babel}
\usepackage{amssymb}
\usepackage{graphics}
\usepackage{epsfig}
\usepackage{amscd}
\usepackage{amsfonts}
\usepackage{amsbsy}
\usepackage{float}
\usepackage{booktabs}
\usepackage{mathrsfs}
\usepackage{caption}
\usepackage{subcaption}
\usepackage{overpic}
\usepackage{dsfont}
\newcommand{\R}{\mathds{R}}
\newcommand{\C}{\ensuremath{\mathbb{C}}}

\newtheorem {theorem} {Theorem}
\newtheorem {prop} [theorem] {Proposition}

\newtheorem {lemma} [theorem] {Lemma}

\newtheorem {remark} {Remark}

\newtheorem*{miTeorema}{Teorema}
\newtheorem{theorem A}{Theorem A}
\newtheorem{theorem B}{Theorem B}
\newtheorem{theorem C}{Theorem C}

\usepackage{color}

\usepackage{fancyhdr}
\begin{document}

\title[Periodic solutions in a tumor-immune competition system with time-delay and chemotherapy effects]{Periodic solutions in a tumor-immune competition system with time-delay and chemotherapy effects}

\author[Pablo Amster, John A. Arredondo, Andr\'es Rivera]
{Pablo Amster, Andr\'es Rivera, John A. Arredondo}

\address{$^1$ Departamento de Matem\'aticas
Universidad de Buenos Aires,
Argentina.}
\address{
$^2$ Departamento de Ciencias Naturales y Matem\'aticas
Pontificia Universidad Javeriana Cali, Facultad de Ingenier\'ia y Ciencias,
Calle 18 No. 118-250 Cali, Colombia.}
\address{
$^3$ Departamento de Matemáticas Fundación Universitaria Konrad Lorenz, Facultad de Ciencias e Ingeniería, Cra 9 bis 62-43, Bogotá-Colombia.}

\email{pamster@dm.edu.ar, amrivera@javerianacali.edu.co, alexander.arredondo@konradlorenz.edu.co}\emph{}



\date{}
\dedicatory{}
\maketitle

\begin{center}\rule{0.9\textwidth}{0.1mm}
\end{center}
\begin{abstract}
The main purpose of this paper is to analyze the dynamics of the system of time-delay differential equations (DDEs)
\begin{equation*}
\begin{split}
\dot{T}(t)&=T(t) f(t,T(t))-\gamma E(t)T(t),\\
\dot{E}(t)&=\sigma+ \frac{pE(t)T(t-\tau_1)}{g+a T(t-\tau_1)}-\frac{mE(t)T(t-\tau_2)}{g+a T(t-\tau_2)}-\eta E(t),
\end{split}
\end{equation*}
where $T=T(t)$ and $E=E(t)$  represent the concentrations of tumor and effector cells at the time $t$. The coefficients $\sigma$, $\mu$, $\gamma$, and $\eta$ are all positive, and $f(t, T)$ represents the relative growth rate of tumor cells, corresponding to a generalized logistic growth function that describes periodic time chemotherapeutic effects. The parameter $\tau_1 \in \mathbb{R}_{\ge 0}$ is the response time delay of the immune system (mediated by effector cells) to an invasion of tumor cells, while $\tau_2 \in \mathbb{R}_{\ge 0}$ represents the time delay of tumor cells in response to the appearance of effector cells.
\end{abstract}


\section{Introduction}

A tumor is an abnormal growth of body tissue that the immune system can normally control in most cases. The formation of a tumor often occurs due to an immune system issue in the body \cite{friedman2012}.  
Many works in the literature describe tumor growth, e.g.,
The tumor growth curves are
classically exponential, logistic, and Gompertz using a population approach.  In analyzing the Gompertz model, several studies have reported a striking correlation between the two parameters of the model, which could be used to reduce dimensionality and improve predictive power. Candidate models of tumor growth included exponential, logistic, and Gompertz models. The exponential and more notable logistic models failed to describe the experimental data, whereas the Gompertz model generated better fits \cite{Murphy}.

Several mathematical models have been developed for the interactions between the immune system and a growing tumor. In \cite{Jianquan} for $T:=T(t)$ and $E:=E(t)$  functions depending of time representing the concentrations of tumor and effector cells at time $t$, the authors consider a logistic relative growth rate, i. e.  $f(T)=r\big(1-\frac{T}{K}\big)$, where $r$ and $K$ are the maximal growth rate and the carrying capacity of the biological environment of the tumor cells, the following system is obtained: 
\begin{equation}\label{main eq 2}
\begin{split}
\dot{T}&=rT \Big(1-\frac{T}{K}\Big)-\gamma E(t)T(t),\\
\dot{E}&=\sigma+ \frac{pE(t)T(t-\tau)}{g+a T(t-\tau)}-E(t)\big(mT(t)+\eta\big).
\end{split}
\end{equation}
This manuscript  considers the following periodic tumor-immune system with delays inspired by the work of J. Li et. al. \cite{LI} 
\begin{equation}\label{main eq 3}
\begin{split}
\dot{T}(t)&=T(t) f(t,T(t))-\gamma E(t)T(t),\\
\dot{E}(t)&=\sigma+ \frac{pE(t)T(t-\tau_1)}{g+a T(t-\tau_1)}-\frac{mE(t)T(t-\tau_2)}{g+a T(t-\tau_2)}-\eta E(t),
\end{split}
\end{equation}

\begin{table}[h]
\centering
\caption{Description of parameters of Eq. $\eqref{main eq 3}$.}
\renewcommand{\arraystretch}{1.3} 
\begin{tabular}{@{}p{4.5 cm}p{10cm}@{}}
\toprule
\textbf{Variables/Parameters} & \textbf{Interpretation} \\
\midrule
$T(t)$ & Concentrations of tumor cells ($T$-cells) at time $t$\\
$E(t)$ & Concentrations of effector cells ($E$-cells) at time $t$\\
$T f(t,T)$ & Time variable growth rate function of the tumor cells\\
$\gamma$ & Rate at which effector cells damage (kill) the tumor cells\\
$\sigma$ & Constant input of effector cells\\
$p E T/(g+aT)$ & Recruiment term of $E$-cells by stimulation of $T$-cells\\
$p$ & Maximum recruitment coefficient of $E$-cells
cells by stimulation of $T$-cells\\
$m E T/(g+aT)$ & Neutralization term of $E$-cells by stimulation of $T$-cells\\
$m$ & Maximum decline rate of $E$-cells by $T$-cells \\
$g,a$ & Parameters associated to the Holling's type functional response $\frac{T}{g+aT}$. $g^{-1}$-rate attack. $a$-handling time\\
\bottomrule
\end{tabular}
\label{table:parameters Jianquan model}
\end{table}

for which, in the absence of effector cells, the relative growth rate of tumor cells is given by the time-periodic Richards-type equation, see \cite{TSOULARIS200221}
\[
f(t,T)=r\Big[\frac{\big(1-T^{\beta}\big)}{\beta}-b(t)\Big],
\]
where $r>0$ is the Malthusian parameter of the tumor cells, $b(t)\in C(\mathbb{R}/q\mathbb{Z})$ (continuous and $q$-periodic function) that represents chemotherapeutic effects on the tumor cell mass, see \cite{Panetta} and $0<\beta \leq 1$ is a real parameter that provides well-known growth functions like
\begin{equation*}
\begin{split}
\text{Logistic relative growth rate ($\beta=1$)} \quad \Rightarrow \quad f(r,T)=r\big[(1-T)-b(t)\big],\\
\text{Gompertz relative growth rate ($\beta \searrow 0$) \, $\Rightarrow$} \quad f(r,T)=r\big[\ln{\frac{1}{T}}-b(t)\big].\\
\end{split}
\end{equation*}
Notice that, without loss of generality, we assumed that the carrying capacity of the biological environment of the tumor cells is equal to $1$.

In this study, we focus mainly on the effects of time delays on the dynamics of tumor and effector cell concentrations and the existence of periodic solutions of \eqref{main eq 3} as a continuation of specific equilibrium solutions. The rest of the paper is organized into five sections and one appendix.  In Section \ref{secTB} we introduce some well-known results on the existence and stability of equilibria solutions and periodic solutions for delayed systems. In this section \ref{section Basic results}, we obtain the equilibrium points for equation \ref{main eq 3} in terms of the different parameters associated with this system. In sections \ref{MainResults} and \ref{continuacion}  we present our main results: First, we provide an explicit interval of delay parameters for the local stability of two of the equilibrium points in the delayed autonomous. Secondly, we state
and prove two results about the existence of periodic solutions of \eqref{main eq 3} via the continuation method adapted to our work. In Section  \ref{NV}  we perform some computational results, often using graphs and phase portraits, to visually validate the theoretical findings from the stability and bifurcation analysis.

\section{Theoretical Background}\label{secTB}
This section presents the theoretical
background for the main aspects discussed in this document. Firstly,  we analyze the linear stability of the equilibria of  \eqref{main eq 3} together with the stability switch curves in the plane $(\tau_1,\tau_2)\in \mathbb{R}^2_{+}$.  Secondly, we provide a theoretical scenario that will allow us to study the existence of $\omega$-periodic solutions of \eqref{main eq 3} considered as a perturbation of an autonomous nondelayed system.

\textbf{On the stability of equilibria for delayed linear systems with two delays}. We finish this section with the following general result on the following linear delay system
\begin{equation}\label{linear system with one delay}
\dot{Y}=A_{0}Y+A_1Y(t-\tau_1)+A_2Y(t-\tau_2),
\end{equation}
with  $A_0,A_1,A_2 \in \mathbb{M}_{2\times 2}$ constant real matrices, $Y\in \mathbb{R}^{2}$ and $\tau_1,\tau_2 \geq 0$. Assume that $\displaystyle{\det A_1=\det A_2=0}$. Following the results in Appendix B, the associated characteristic equation is given by:
\begin{equation}\label{characteristic new general}
\mathcal{P}(\lambda,\tau_1,\tau_2)=0, \quad \text{with} \quad \mathcal{P}(\lambda,\tau_1,\tau_2)=p_{0}(\lambda)+p_{1}(\lambda)e^{-\lambda \tau_1}+p_{2}(\lambda)e^{-\lambda \tau_2}+p_{3}(\lambda)e^{\lambda (\tau_1+\tau_2)}
\end{equation}
where
\begin{equation}\label{coeficientes}
\begin{split}
p_{0}(\lambda)&=\lambda^{2}-(tr A_0)\lambda+\det A_0,\\
p_{1}(\lambda)&=C_{A_0A_1}-(tr A_1)\,\lambda, \hspace{2.0 cm} \text{and} \quad \quad  C_{A_{i}A_{j}}=\det{(A^{1}_{i}|A^2_{j})} + \det{(A_{j}^1|A_{i}^2)}, \\
p_2(\lambda)&=C_{A_0A_2}-(tr A_2)\,\lambda,\\
p_3(\lambda)&=C_{A_1A_2},
\end{split}
\end{equation}
Here, $\displaystyle{\det(a^1|b^2)}$ denotes the matrix with the first column and the second column of the matrices. 
\[
A=[a^1,a^2] \quad \text{and} \quad B=[b^1,b^2],
\]
respectively. Notice that the analytic function $p_{i}(\lambda)$ satisfies the following
\begin{itemize}
    \item[a.] $\text{deg}(p_{0}(\lambda))=2>1\geq\text{deg}(p_{i}(\lambda))$ with $i=1,2,3.$
     \item[b.] $p_i(\lambda)$ have no common factors.
     \item[c.] For all $s\in \mathbb{R}$, $p_{l}(is) \neq 0$, for $l=1,2,3$.
    \item[d.] On $\C$ we have
    \[
    \lim_{|\lambda| \rightarrow \infty} \sup \left(\left|\frac{p_1(\lambda)}{p_0(\lambda)}\right|+\left|\frac{p_2(\lambda)}{p_0(\lambda)}\right|+\left|\frac{p_3(\lambda)}{p_0(\lambda)}\right|\right)=0.
    \]
\end{itemize}
We need the following condition
\begin{itemize}
\item[H1)] Finite number of characteristic roots on $\C_{+}=\left\{\lambda \in \C: \text{Re}(\lambda)>0\right\}$
    \item[H2)] $\lambda=0$ is not a characteristic root for any $\tau_1, \tau_2$, i.e., $p_0(0)+p_1(0)+p_2(0)+p_{3}(0)\neq 0$.
\end{itemize}
In the following, we present some well-known results on the solutions $\lambda=\lambda(\tau_1,\tau_2)$ of \eqref{characteristic new general}. The simplest case emerges when both delays are zero. Direct and easy computations prove the following. 
\begin{lemma}
For $\tau_1=\tau_2=0$, the roots $\lambda=\lambda(0,0)$ satisfy the quadratic equation
\[
P(\lambda,0,0)=\lambda^{2}-tr\mathcal{A}\,\lambda+\det \mathcal{A},
\]
with $\mathcal{A}=A_0+A_1+A_2$. Moreover, if
\[
tr \mathcal{A}<0 \quad \text{and} \quad \det \mathcal{A}>0,
\]
the trivial solution $Y(t)=\textbf{0}$ of \eqref{linear system with one delay} is locally asymptotically stable. Meanwhile, if
\[
tr \mathcal{A}>0, \quad \text{or} \quad tr \mathcal{A}=0 \quad \text{and} \quad \det \mathcal{A}<0.
\]
hold, is unstable. 
\end{lemma}

\textbf{On the stability switching curves.} In the next lines, we present some well-known results about the stability switching curves of delayed systems with two
delays and delay-independent coefficients; see \cite{Lin}
\begin{lemma} Assume $\tau_1,\tau_2 \in \mathbb{R}_{+}$. As $(\tau_1,\tau_2)$ varies continuously in $\mathbb{R}^2_{+}$ the number of characteristic roots (including their multiplicities) of $P(\lambda,\tau_1,\tau_2)$ on $\C_{+}$, can change only if a characteristic roots appear on or cross the imaginary axis.
\end{lemma}

From this result, we are looking for solutions of \eqref{characteristic new general} of the form $\displaystyle{\lambda=i\vartheta}$ with $\vartheta \in \mathbb{R}_{+}$

\textbf{Periodic solutions.} The end of this section is devoted to establish a general abstract result for the existence of $\omega$-periodic solutions of the system
\begin{equation}
\label{sys-impl}
X'(t)= \Phi(X(t),y),
\end{equation}
where the parameter $y$ belongs to an appropriate Banach space $Y$. 
In more precise terms, let $C_\omega$ denote the set of continuous $\omega$-periodic functions $X:\R\to\R^n$ and $C_\omega^1:=C_\omega\cap C^1(\R,\R^n)$.  
Assume that $\mathcal O\subset C_\omega\times Y$ is open and $\Phi:\mathcal O\to C_\omega$ is a continuously Fr\'echet differentiable operator, then the operator $\mathcal F:C_\omega^1\times Y\to \R^n$ given by 
$$\mathcal F(X, y):= X' - \Phi(X,y)$$ is well defined and continuously differentiable. For instance, with system (\ref{main eq 3}) in mind, we may set
$b(t):= b_0 + \varphi(t)$ with $\varphi\in C_\omega$,   $\tau:= (\tau_1,\tau_2)$ and  
$Y:=C_\omega\times \R^2$. The right-hand side of the system, which is generically written as $F(t,X(t),X(t-\tau_1),X(t-\tau_2))$, can be equivalently expressed in the form $\Phi(X,\varphi,\tau)$ with the help of the continuous operators  
 $X\mapsto X_{\tau_j}$ 
  defined on $C_\omega$ by $X_{\tau_j}(t):= X(t-\tau_j)$ for $j=1,2$. 
When $y=(\varphi, \tau)=0$, the system is autonomous 
and has no delay; given an equilibrium point $X_0$, our goal consists in finding a smooth branch of $\omega$-periodic solutions $X=X(y)$ such that
$X(0)=X_0$.  The following result is a direct consequence of the Implicit Function Theorem, combined with the Fredholm alternative: 
\begin{theorem}\label{thm-impl}
    In the previous context, let $X_0$ be a solution of the system $\mathcal F(X,0)=0$
 and assume that the linear system 
$$Z'(t)= D_X\Phi(X_0,0)Z(t),
$$
where    $D_X$ denotes  the Fr\'echet partial derivative  with respect to $X$, 
has no nontrivial $\omega$-periodic solutions. 
Then there exist open sets $\mathcal V,\mathcal W$ with   $0\in \mathcal V\subset Y$, $X_0\in \mathcal W\subset C_\omega$ and a unique smooth function $X:\mathcal V\to\mathcal W$ such that $X(0)=X_0$ and $X(y)$ is  a solution of (\ref{sys-impl}) for each $y\in\mathcal V$. 
\end{theorem}
Focusing on the application 
to (\ref{main eq 3}), it shall be useful to recall the following simple criterion:
\begin{lemma}\label{iso}
    Let $M\in \R^{n\times n}$ and consider the linear operator $L:C^1_\omega\to C_\omega$ given by $LZ:=Z' - MZ$. 
    Then $\ker(L)=\{0\}$ if and only if 
    $\pm \frac {2k\pi}\omega i$ is not an eigenvalue of $M$ for any 
$k\in\mathbb N_0$.
\end{lemma}



\section{Basic Results}\label{section Basic results}
In this section, we study some basic results on the system \eqref{main eq 3}. Firstly, let us define $X:=(T,E)^{tr}$. Under this notation,
\[
X(t)=(T(t),E(t))^{tr}, \quad X_{\tau_1}(t):=X(t-\tau_1), \quad X_{\tau_2}(t):=X(t-\tau_2), 
\]
is a solution of \eqref{main eq 3} if $X=X(t)$ satisfies the non-autonomous delayed system 
\begin{equation}\label{delayed system}
\dot{X}(t)=F(t,X(t),X_{\tau_1}(t),X_{\tau_2}(t)),
\end{equation}
 with $F(t,X(t),X_{\tau_1}(t),X_{\tau_2}(t))$ given by
\begin{equation}\label{vector field}
F(t,X,X_{\tau_1},X_{\tau_2})=\begin{pmatrix}
    T\big(f(t,T)-\gamma E\big) \\
    \sigma+ E\big( ph (T_{\tau_1})-m h(T_{\tau_2})-\eta\big)
\end{pmatrix} \quad \text{with} \quad h(s)=\frac{ s}{g+a s}.
\end{equation}
We prove this for our system \eqref{delayed system} subject to the initial condition.
\[
\begin{split}
T(s)&=\Psi_1(s), \quad E(s)=\Psi_2(s),\\
\Psi_1(s),\Psi_2(s)&\geq 0, \,s \in [-\tau,0], \quad \Psi_1(0),\Psi_2(0)\geq 0, 
\end{split}
\]
where $\Psi(s):=(\Psi_{1}(s),\Psi_{1}(s))\in C([-\tau,0],\mathbb{R}^{2}_{+})$ the Banach space of continuous functions mapping $[-\tau,0]$ to $\mathbb{R}^{2}_{+}$ with  $\tau=\max\left\{\tau_1,\tau_2\right\}$,  $\|\Psi\|=\sup_{-\tau \leq s\leq 0}|\Psi_i(s)|
$ and $\mathbb{R}^{2}_{+}=\left\{(x_{1},x_{2})|\,x_i\geq 0)\right\}$. In addition, \eqref{vector field} is a Lipschitz continuous function, then, locally, there exists a unique solution $X=X(t,\Psi)$ for any continuous initial function $\Psi.$

\subsection{Two fundamental properties} Let us start with an analogous result of Lemma 1 and Lemma 5 in \cite{Galach2003} about the positivity and maximal interval of existence of solutions with non-negative initial conditions and the non-existence of non-negative periodic orbits in \eqref{main eq 3} for $\tau_1, \tau_2\geq 0.$
\begin{theorem} Any solution $X(t,\Phi)$ of \eqref{main eq 3} globally defined for all $\Phi\in C([-\tau,0],\mathbb{R}^{2}_{+}) $ and is positive, i.e., $T(t)\geq 0$ and $E(t)\geq 0$ for all $t\in  \mathbb{R}.$
\end{theorem}
\begin{proof}
From the existence and uniqueness theorem, for each $X_0=(\Psi_{1}(0),\Psi_{2}(0))$, with $\Psi_{1}(0),\Psi_{2}(0) \geq 0$, there exists a unique solution $X(t)$ of \eqref{main eq 3} satisfying $X(0)=X_{0}$, defined in some interval $(0,\omega_{+})$ with $\omega_{+} \leq \infty$. Let us define the function
\[
\mathcal{K}(t):=\int_{0}^{t}(f(s,T(s))-\gamma E(s))\,ds.
\]
A direct computation of the first equation of \eqref{main eq 3} shows that
\[
\frac{d}{dt}(\exp[-\mathcal{K}(t)]T(t))=\exp[-\mathcal{K}(t)]\big(-\mathcal{K}(t)T(t)+\mathcal{K}(t)T(t)\big)=0 \quad \Rightarrow\quad T(t)=\Psi(0)\exp\Big[\mathcal{K}(t)\Big].
\]
That is, 
\[
T(t)=\Psi(0)\exp\big[\int_{0}^{t}(f(s,T(s))-\gamma E(s))\,ds,
\]
showing that $T(t)\geq 0$ for all $t\in (0, \omega_{+})$. Similarly, from the second equation of \eqref{main eq 3} we have
\[
\dot{E}(t)\geq E\big( ph (T_{\tau_1})-m h(T_{\tau_2})-\eta\big),
\]
therefore, 
\[
\dot{E}(t)\geq \Psi_2(0)\exp\Big[\int_{0}^t \big( ph (T_{\tau_1})-m h(T_{\tau_2})-\eta\big)\, ds\Big].
\]
and then $E(t)\geq 0$ for all $t\in (0,w_{+})$.
Since $T(t)$ and $E(t)$ are positive and well-defined for all $t\in (0,w_{+})$ for some $w_{+}\leq \infty$, we have 
\[
T^{\prime}(t)\leq T(t)f(t,T(t)) \quad \Rightarrow \quad T(t) \leq T_{M}:=\max \left\{T(0),[(1-\beta b_{\min})/r]^{1/\beta}\right\}.
\]
where $b_{\min}:=\min b(t)$. Coming back to the equation of $E(t)$, we have
\[
E^{\prime}(t)\leq \sigma + \kappa E(t), \quad \text{with} \quad \kappa= \begin{cases}
-\eta \quad \text{if} \quad m\geq p, \\
(p-m) h(T_{M})-\eta  \quad \text{if} \quad m < p, \\
\end{cases}
\]
therefore
\[
E(t)\leq \exp[\kappa t]\Big(E(0)+\sigma \int 
_{0}^{t} \exp[-\kappa s] \, ds\Big). 
\]
The previous computations imply that $\omega_{+}=\infty$, otherwise if $\omega_{+}<\infty$, from the classical results of the continuation of solutions we have $\displaystyle{\lim_{t \to \omega^{-}_{+}} |X(t)|=\infty}$ which is a contradiction.
\end{proof}
\begin{remark}
Alternatively, the global definition of any solution can be established by the following argument:
\[
\dot{E}(t)\leq \sigma+pE(t), 
\]
implying that $E(t)$ globally defined. Then, for $T(t)$ we have
\[
\dot{T}(t)\leq T(t)(\kappa-\gamma E(t)), \quad \text{with} \quad k:=\max_{\mathbb{R}\times \mathbb{R}^{+}}{f(t,T)},
\]
proving that $T(t)$ is globally defined.
\end{remark}
\begin{lemma}
Assume $\tau=(0,0).$ Then, there are no non-negative and non-trivial periodic solutions of the system \eqref{new main EQ aut}.
\end{lemma}
\begin{proof}
For $X=(T,E)^{tr}$, consider the auxiliary scalar function $\displaystyle{L:\mathbb{R}_+^2 \to \mathbb{R}_+}$, $L(X)=\frac{1}{T E}$. Define the vector field $\tilde{F}:\mathbb{R}_+^2 \to \mathbb{R}_+^2$, $X \to \tilde{F}(X)$ given by
\[
\tilde{F}(X)=L(X)F(X)=\begin{pmatrix}
 \frac{f(T)}{E}-\gamma\\
 \frac{\sigma}{TE}+\frac{p}{g+aT}-\frac{m T+\eta}{T}
\end{pmatrix}
\]
Then, the divergence of $\tilde{F}$ is computed as
\[
\nabla \cdot \tilde{F}(X)=-\frac{1}{E}\Big(r\beta T^{\beta-1}+\frac{\sigma}{T E}\Big)<0,
\]
by Dulac-Bendixson Theorem, it follows that \eqref{new main EQ aut} does not admit closed orbits in $\mathbb{R}_+^2$
\end{proof}

Let us begin with some results on the existence and stability of equilibrium points for the associated autonomous system, i.e., when $b(t)=\hat{b}$ for all $t\in \mathbb{R}$, with $\hat{b}\in \mathbb{R}^{+}.$ In this case, we have the autonomous time-delayed system
\begin{equation}\label{Main EQ aut}
\begin{split}
\dot{T}&=T \big(f(T)-\gamma E\big),\\
\dot{E}&=\sigma+ E\big( p\, h(T_{\tau_1})-m \, h(T_{\tau_2})-\eta\big),
\end{split}
\end{equation}
where the relative tumor growth function (for $0<\beta$) is given by
\begin{equation}\label{funcion-crecimiento}
\displaystyle{f(T):=r(b-T^{\beta})},
\end{equation}
 with $r\to r/\beta$ and $b:=1-\beta \hat{b}.$ In addition, 
 from now on we shall assume that $b>0$, that is:
 \[
  \hat{b}<1/\beta \quad \text{for any $\beta>0$}.
 \]

\subsection{Equilibrium points} The possible equilibrium points $(T,E)$ of \eqref{Main EQ aut} are given by the solutions of the following system
\begin{equation}\label{equilibrium system}
\begin{split}
T\big(f(T)-\gamma E)&=0,\\
\sigma+E\Big(\frac{(p-m)T}{g+aT}-\eta\Big)&=0.
\end{split}
\end{equation}
We are only interested in solutions of \eqref{equilibrium system} lying in the first quadrant of the $TE$-plane. i.e., points $(T,E)$ satisfying $T\geq 0$ and $E\geq 0$. Any equilibrium point that fulfills this condition shall be called an \textit{admissible equilibrium point}. The most simple equilibrium point is obtained when $T=0$. Indeed, in this case, we have that $(0,\sigma/\eta)$ is an equilibrium point (for any function $f(T)$), known as \textit{tumor-free equilibrium point}.

\vspace{0.3 cm}
Otherwise, for $T\neq 0$, any solution of \eqref{equilibrium system} must satisfy
\begin{equation}\label{nontrivial equilibrium systems}
f(T)=\gamma E, \quad \text{and} \quad \sigma=E\big(\eta+\frac{(m-p) T}{g+aT}\big).
\end{equation}
From the left-hand equation in \eqref{nontrivial equilibrium systems} we deduce that $f(T)\geq 0$. Therefore, from the formula (\ref{funcion-crecimiento}), any solution $(T,E)$ of \eqref{nontrivial equilibrium systems} satisfies  $T\in [0,b^{1/\beta}]$. Easy computations show that
\[
\begin{split}
\sigma \gamma&=f(T)\big(\eta+\frac{(m-p)T}{g+aT}\big),\\
&=\eta f(T)+\frac{(m-p)Tf(T)}{g+aT},
\end{split}
\]
from which it follows that 
the solutions of the equation
$$
\eta f(T)=\sigma\gamma -\frac{(m-p)Tf(T)}{g+aT},$$
namely \begin{equation}\label{equation for equilibria}
\eta r(b-T^{\beta})=\sigma \gamma-\Big(\frac{m-p}{g+aT}\Big)rT(b-T^{\beta})
\end{equation}
provide the first component of possible non-tumor-free equilibrium points of system \eqref{Main EQ aut}.
To avoid parameter overload, we introduce the scaling given by
\begin{equation}\label{parametros-escalados}
\sigma \to \frac{\sigma \gamma}{\eta r}, \quad   m\to \frac{m}{\eta g}, \quad  p\to \frac{p}{\eta g}, \quad a\to \frac{a}{g} \quad \text{and call} \quad \mu:= m-p,
\end{equation}

After some direct and easy computations, \eqref{equation for equilibria} can be rewritten as follows
\begin{equation}\label{equation for S}
S(T,\mu,a)=b-\sigma, \quad \text{with} \quad S(T,\mu,a)=T[(a+\mu)T^{\beta}+T^{\beta-1}-(b(\mu+a)-a\sigma)]
\end{equation}

In conclusion, any possible admissible equilibrium point $(T,E)$ of \eqref{Main EQ aut} that is not the tumor-free equilibrium point, satisfies 
\[
0< T\leq b^{1/\beta}, \quad E=f(T)/\gamma, \quad \text{and} \quad S(T,\mu,a)=b-\sigma.
\]
\begin{remark}
The next results consider the case $a=0.$ Following the section of numerical simulations in \cite{KHAJANCHI2014652} the parameter $g$ is of the order $10^6$ (pretty large) with respect to the other parameter in the model. Then, after the scaling process, the value $a\to a/g=\textit{O}(10^{-6})$, is quite small.
\end{remark}
\vspace{0.2 cm}
For $a=0$, the function $S(T,\mu,0)$ is given by
\[
S(T,\mu,0)=\mu(T^\beta-b)T+T^{\beta}.
\]

Appendix A provides mathematical properties of $S(T,\mu,0)$ useful for obtaining the following results on the equation \eqref{equation for S} in this case.


\begin{prop}\label{Prop 1}
Let $\beta, b \in\, (0,1)$ and let us assume $a=0$, in addition to the following conditions
\begin{equation}\label{global condition for equilibria}
b\leq \sigma, \quad \text{and} \quad \mu\leq \mu_{c}:=\Big(\frac{1}{b}\Big(\frac{1-\beta}{1+\beta}\Big)^{\beta-1}\Big)^{\frac{1}{\beta}}.
\end{equation}
Then, equation \eqref{equation for S} has no solutions lying in $[0,b^{1/\beta}]$ i.e., the tumor-free equilibrium point $(0,\sigma/\eta)$ is the only admissible equilibrium point of \eqref{Main EQ aut}. 
\end{prop}

\begin{proof}
As we know, if $(T,E)$ is any admissible equilibrium point of \eqref{Main EQ aut} besides the tumor-free equilibrium point, then it satisfies 
\[
S(T,\mu,a)=b-\sigma, \quad \text{with} \quad 0< T\leq b^{1/\beta}.
\]
If $a=0$, we have $\displaystyle{S(T,\mu,0)=\mu (T^{\beta}-b)T+T^{\beta}}$. The result follows directly from Lemma \ref{lema function h for mu negativo} and Lemma \ref{lema function h for mu positive 1} in Appendix A. 
\end{proof}

\begin{prop}\label{Prop unique nontrivial equilibrium}
Let $\beta, b \in\, (0,1)$ and assume that $a=0$. 
Suppose that one of the following conditions holds
\[
\dagger) \quad \sigma <b \quad \text{and} \quad  \mu\leq 0, \quad \text{or} \quad \ddagger) \quad \sigma\leq b \quad \text{and} \quad 0<\mu\leq \mu_{c}.
\]
Then, equation \eqref{equation for S} has exactly one solution $T_{\ast}$ on $]0,b^{1/\beta}]$ (which is simple), given implicitly as the solution of
\begin{equation}\label{ecuacion para T unico}
\mu(T^{\beta}_{\ast}-b)T_{\ast}+T^{\beta}_{\ast}=b-\sigma.
\end{equation}
Consequently, the only admissible equilibrium points of \eqref{Main EQ aut} are $(0,\sigma/\eta)$ and $\displaystyle{(T_{\ast}, E_{\ast})}$ with $\displaystyle{E_{\ast}:=r(b-T^{\beta}_{\ast})/\gamma.}$ 
\end{prop}
\begin{proof}
Let $h_{0}:=b-\sigma$, then $0\leq h_{0}\leq b.$ The proof follows as a direct consequence of Lemma \ref{lema function h for mu negativo} case 2a. and Lemma \ref{lema function h for mu positive 1} case b. in Appendix A. 
\end{proof}
\begin{remark}
In the case  $\beta=1$, we have
\[
\begin{split}
T_{\ast}=b-\sigma, \quad \text{and} \quad E_{\ast}=\frac{r \sigma}{\gamma} \quad \text{if} \quad \mu=0, 
\end{split}
\]
\[
T_{\ast}=\frac{\mu b-1+\sqrt{(1+\mu b)^2-4\mu \sigma}}{2\mu}, \quad \text{and} \quad E_{\ast}=\frac{r}{2\gamma}\Big(b+\frac{1-\sqrt{(1+\mu b)^2-4\mu \sigma}}{\mu}\Big) \quad \text{if} \quad \mu<0.
\] 
In particular,
\[
T_{\ast}=\big(1-\sqrt{\frac{\sigma}{b}}\big)b, \quad \text{and} \quad E_{\ast}=\frac{r\sqrt{\sigma b}}{\gamma}, \quad \text{if} \quad \mu=-b^{-1}.
\]
 \end{remark}
\vspace{0.2 cm}
\noindent
On the other hand, for each fixed $\mu>\mu_{c}$, let $T_{L,R}(\mu)$ be the two solutions of
\[
T^{\beta-1}(\beta+\mu(1+\beta)T)=\mu b, \quad \text{with} \quad 0< T_{L}(\mu)< \frac{\mu(1-\beta)}{1+\beta} < T_{R}(\mu)<b^{1/\beta},
\]
and define the following quantities
\[
H_{L}:=\mu(T^\beta_{L}(\mu)-b)T_{L}(\mu)+T_{L}^{\beta}(\mu), \quad \text{and} \quad H_{R}:=\mu(T^\beta_{R}(\mu)-b)T_{R}(\mu)+T_{R}^{\beta}(\mu),
\]
In Appendix A it is shown that $H_{R}<H_{L}$ and, in particular, if $\displaystyle{\mu\geq \mu_{bif}:=\frac{1}{\beta \big[b(1-\beta)^{1-\beta}\big]^{1/\beta}} }$, then $\displaystyle{H_{R}\leq 0}$ with
\[
H_{R}=0 \quad \Leftrightarrow \quad T_{R}(\mu_{bif})=[(1-\beta)b]^{1/\beta}.
\]
\begin{remark}
Direct computation shows that for all $\displaystyle{]0,1]}$ we have
\[
\begin{split}
\mu_{c}<\mu_{bif}, \quad \text{and} \quad \mu_{c} \to \frac{1}{b}, \quad \mu_{bif}\to \frac{1}{b} \quad \text{as} \quad \beta \to 1 
\end{split}
\]
\[
T_{\ast}=\frac{\mu b-1+\sqrt{(1+\mu b)^2-4\mu \sigma}}{2\mu}, \quad \text{and} \quad E_{\ast}=\frac{r}{2\gamma}\Big(b+\frac{1-\sqrt{(1+\mu b)^2-4\mu \sigma}}{\mu}\Big) \quad \text{if} \quad \mu<0.
\] 
In particular,
\[
T_{\ast}=\big(1-\sqrt{\frac{\sigma}{b}}\big)b, \quad \text{and} \quad E_{\ast}=\frac{r\sqrt{\sigma b}}{\gamma}, \quad \text{if} \quad \mu=-b^{-1}.
\]
 \end{remark}

\begin{prop}\label{Prop 3} Let $\beta, b \in\, ]0,1[$, $\displaystyle{h_{0}:=b-\sigma}$, and let us assume $a=0$. 

\begin{itemize}
    \item[1.] For $\displaystyle{\mu_{c}<\mu \leq \mu_{bif}}$, it holds: 
\begin{itemize}
    \item[I)] If $ h_{0}=H_R$, then there are exactly two solutions of \eqref{equation for S}  on $[0,b^{1/\beta}]$.
    \item[II)] If $H_{R}<h_{0}< H_{L}$, then there are exactly three solutions of \eqref{equation for S} on $[0,b^{1/\beta}]$.
    \item[III)]  If $H_{L}=h_{0}$, then there are exactly two solutions of \eqref{equation for S} on $ [0,b^{1/\beta}]$.
    \item[IV)]  If $0\leq h_{0}< H_R$ or $H_{L}< h_0$, then there is exactly one solution of \eqref{equation for S} on $[0,b^{1/\beta}]$.
\end{itemize}

\hspace{0.3 cm}
\item[2.] For $\displaystyle{\mu>\mu_{bif}}$, it holds:
\begin{itemize}
    \item[I)] If $ h_{0}=H_R$, then there are exactly one solution of \eqref{equation for S}  on $[0,b^{1/\beta}]$.
    \item[II)] If $H_{R}<h_{0}\leq  0$, then there are exactly two solutions of \eqref{equation for S} on $[0,b^{1/\beta}]$.
    \item[III)]  If $0<h_{0}<H_L$, then there are exactly three solutions of \eqref{equation for S} in $ [0,b^{1/\beta}]$.
    \item[IV)] If $ h_{0}=H_L$, then there are exactly two solutions of \eqref{equation for S} in $[0,b^{1/\beta}]$.
    \item[V)]  If $h_{0}< H_R$ there are no solutions pf \eqref{equation for S}, while if $H_{L}< h_0$, there is exactly one solution of \eqref{equation for S} on $[0,b^{1/\beta}]$.
\end{itemize}

\end{itemize}
\end{prop}


\section{Main results}\label{MainResults}
This section presents some results about the dynamics and stability of the feasible equilibrium points for the autonomous delayed system \eqref{Main EQ aut}, given by
\begin{equation}\label{new main EQ aut}
\dot{X}=\mathcal{F}(X,X_{\tau_1},X_{\tau_2}) \quad \text{with} \quad 
\mathcal{F}(X_{\tau_0},X_{\tau_1},X_{\tau_2})= \begin{pmatrix}
T\big(f(T)-\gamma E\big)\\
\sigma+ E\big( ph (T_{\tau_1})-m h(T_{\tau_2})-\eta\big)
\end{pmatrix},
\end{equation}
where $X_{\tau_i}=(T_{\tau_i},E_{\tau_i})^{tr}$. For convenience, here we denote $X_{\tau_0}=X=(T,E)$, that is, $\tau_0=0$. Notice that at each equilibrium point $\hat{X}=(\hat{T},\hat{E})$, the associated linear system is given by
\begin{equation}\label{linearized system with delay}
\dot{Y}=A_0Y+A_1Y_{\tau_1}+A_2Y_{\tau_2},
\end{equation}
with $A_{i}=D_{X_{\tau_i}}\mathcal{F}(\hat{X})$. Moreover,
\[
A_{0}=\begin{pmatrix}
f(\hat{T})-\gamma \hat{E}+\hat{T}f^{\prime}(\hat{T}) & -\gamma\hat{T}\\
0 & (p-m)h(\hat{T})-\eta
\end{pmatrix},
\]
and
\[
A_1=\begin{pmatrix}
0 & 0\\
p\, h^{\prime}(\hat{T})\hat{E}& 0
\end{pmatrix}, \qquad A_2=\begin{pmatrix}
0 & 0\\
-m\, h^{\prime}(\hat{T})\hat{E} & 0
\end{pmatrix}.
\]
From Appendix B, the corresponding characteristic function $\mathcal{P}_{a}(\lambda,\tau_1,\tau_2)$ of \eqref{linearized system with delay} is given by
\begin{equation*}\label{characteristic polynomial}
\mathcal{P}_{a}(\lambda,\tau_1,\tau_2):=P(\lambda)+e^{-\tau_1\lambda} Q(\lambda) + e^{-2\tau_1\lambda}\det A_1+e^{-\tau_2\lambda} H(\lambda,\tau_1)+e^{-2\tau_2\lambda}\det A_2,
\end{equation*}
with
\begin{equation}\label{coeficientes}
\begin{split}
P(\lambda)&=\lambda^{2}-(tr A_0)\lambda+\det A_0\\
Q(\lambda)&=C_{A_0A_1}-(tr A_1)\,\lambda, \hspace{2.0 cm} \text{and} \quad \quad  C_{A_{i}A_{j}}=\det{(A^{1}_{i}|A^2_{j})} + \det{(A_{j}^1|A_{i}^2)}, \\
H(\lambda,\tau_{1})&=e^{-\lambda\,\tau_1}C_{A_1A_2}+C_{A_0A_2}-(tr A_2)\,\lambda,
\end{split}
\end{equation}
where, $A^{1}_{i}$, $A^{2}_{i}$ denoted the first and the second column of $A_i$ respectively. Direct computations show that
\[
\begin{split}
tr A_0&=f(\hat{T})-\gamma \hat{E}+\hat{T}f^{\prime}(\hat{T})+(p-m)h(\hat{T})-\eta, \quad tr A_1=tr A_2=0\\
\det A_0&=\big(f(\hat{T})-\gamma \hat{E}+\hat{T}f^{\prime}(\hat{T})\big)\big((p-m)h(\hat{T})-\eta\big), \quad \det A_1=\det A_2=0.
\end{split}
\]

and
\[
C_{A_0A_{1}}=\gamma p h^{\prime}(\hat{T})\hat{T}\hat{E}, \quad C_{A_0A_{2}}= -\gamma m h^{\prime}(\hat{T})\hat{T}\hat{E},\quad \text{and} \quad C_{A_1A_2}=0.
\]

In consequence, 
\begin{equation}\label{characteristic polynomial new}
\mathcal{P}_{a}(\lambda,\tau_1,\tau_2)=P(\lambda)+\gamma h^{\prime}(\hat{T})\hat{T}\hat{E}\big(p e^{-\tau_1 \lambda}-m e^{-\tau_2 \lambda}\big)
\end{equation}


\subsection{Stability of the tumor-free equilibrium point}
For any  $a \in \mathbb{R}^{+}$,   the characteristic equation \eqref{characteristic polynomial new}  
at the tumor-free equilibrium point $(0,\sigma/\eta)$
is given explicitly by
\begin{equation}\label{tumor-free characteristic equation}
P_{a}(\lambda,\tau_1)=0 \quad \Leftrightarrow \quad (\lambda+\eta)(\lambda-rb+\frac{\gamma \sigma} {\eta})=0.
\end{equation}

\begin{theorem} Define the parameter $\displaystyle{\Delta:=\gamma \sigma-rb\eta}$. Then, for all $\tau_{1},\tau_2 \geq 0$, it follows
\begin{itemize}
    \item[1.] If $\Delta <0$ then the equilibrium $(0,\sigma/\eta)$ is unstable.
    \item[2.] If $\Delta =0$ then $(0,\sigma/\eta)$ is locally stable.
    \item[3.]  If $\Delta >0$ then $(0,\sigma/\eta)$ is locally asymptotically stable.
\end{itemize}
\end{theorem}


\subsection{Stability analysis of non-tumor-free equilibrium points} 
In what follows, the stability properties of the equilibrium point $(T_{\ast}, E_{\ast})$ are discussed. Recall that for $a=0$, $(T_{\ast},E_{\ast})$ is the only solution of  \eqref{equation for S}, with $T_\ast$ solution of equation (\ref{ecuacion para T unico}) and  $\displaystyle{E_{\ast}:=r(b-T^{\beta}_{\ast})/\gamma}$. With this, we observe the following consequence, predicted by Proposition \ref{Prop unique nontrivial equilibrium}.

\begin{remark}
As $(T_{\ast},E_{\ast})$ is a simple solution of equation \eqref{equation for S} at $a=0$, it is natural the analytical extension of this solution to  $a\geq 0$ small.  Therefore, from now on, we will refer to
$$(T_{\ast,a},E_{\ast,a}),$$
as the equilibrium point for a positive small $a$, and under the remaining set of restrictions given on the parameters by Proposition \ref{Prop unique nontrivial equilibrium}.
\end{remark}

Moreover, the characteristic function for $(T_{\ast,a},E_{\ast,a})$ is given by
\begin{equation}\label{eq-caracteristica}
\mathcal{P}_{a}(\lambda,\tau_1,\tau_2)=(\lambda-\lambda^{\ast}_{1,a})(\lambda-\lambda^{\ast}_{2,a})+ h^{\prime}\big(T_{\ast,a})f(T_{\ast,a})T_{\ast,a} (pe^{-\tau_1 \lambda}-m e^{-\tau_2 \lambda}\big),
\end{equation}
where
\begin{equation}\label{valores de lambda}
\lambda^{\ast}_{1,a}=T_{\ast,a}f^{\prime}(T_{\ast,a})<0 \quad \text{and} \quad \lambda^{\ast}_{2,a}=(p-m)h(T_{\ast,a})-\eta .
\end{equation}

\vspace{0.3 cm}
Let us define $\displaystyle{\mathscr{D}_{\ast,a}:=\mathcal{P}_{a}(0,\tau_1,\tau_2)}$ which corresponds to
\begin{equation}\label{determinante}
\begin{split}
\mathscr{D}_{\ast,a}&=\lambda^{\ast}_{1,a} \lambda^{\ast}_{2,a}+h^{\prime}\big(T_{\ast,a})f(T_{\ast,a})T_{\ast,a}(p-m),\\
&=(p-m)T_{\ast,a}\big(f^{\prime}(T_{\ast,a})h(T_{\ast,a})+h^{\prime}(T_{\ast,a})f(T_{\ast,a})\big)-\eta T_{\ast,a}f^{\prime}(T_{\ast,a}).
\end{split}
\end{equation}

By the set of scale parameters \eqref{parametros-escalados}, $\lambda^{\ast}_{2,a}$, and $\mathscr{D}_{\ast,a}$ are given explicitly by
\[
\lambda^{\ast}_{2,a}=-\eta\Big(\frac{\mu T_{\ast,a}}{1+aT_{\ast,a}}+1\Big) \quad \text{and} \quad \mathscr{D}_{\ast,a}=-\eta T_{\ast,a}\Big(\Big(\frac{\mu T_{\ast,a}}{1+aT_{\ast,a}}+1\Big)f^{\prime}(T_{\ast,a})+\frac{\mu f(T_{\ast,a})}{(1+aT_{\ast,a})^2}\Big), 
\]
for all $\displaystyle{\tau=(\tau_1,\tau_2)\in \mathbb{R}^{+}}$. In particular, for $\displaystyle{a=0}$ we have 
\begin{equation}\label{Determinante cero}
\lambda^{\ast}_{2,0}=-\eta\big(\mu T_{\ast}+1\big) \quad \text{and} \quad \mathscr{D}_{\ast,0}=-\eta T_{\ast}\big(\big(\mu T_{\ast}+1\big)f^{\prime}(T_{\ast})+\mu f(T_{\ast})\big), 
\end{equation}
where $\displaystyle{T_{\ast}=T_{\ast,0}}$ is the unique solution of \eqref{ecuacion para T unico}. The first result on the stability of $(T_{\ast}, E_{\ast})$ is the following.
\begin{lemma}
Let $\displaystyle{a\geq 0}$ be small and assume that Proposition \ref{Prop unique nontrivial equilibrium} is valid for $a=0$. If $\displaystyle{\mathscr{D}_{0,\ast}<0}$ then the equilibrium $\displaystyle{(T_{\ast,a},E_{\ast,a})}$ is unstable for all $\tau_1,\tau_2\geq 0$.  
\end{lemma}
\begin{proof}
Direct calculations show that
\[
\mathcal{P}_{0}(0,\tau_1,\tau_2)=\mathscr{D}_{0,\ast}<0, \quad \text{and} \quad \lim_{\lambda \to \infty}\mathcal{P}_{0}(0,\tau_1,\tau_2)=\infty, 
\]
for all, for all $\tau_1,\tau_2\geq 0$. Then by the Intermediate Value Theorem, there exists $\displaystyle{\lambda_0>0}$ such that $\displaystyle{\mathcal{P}_{0}(\lambda_0,\tau_{1},\tau_2)=0.}$ By continuity and for $\displaystyle{a>0}$ small, the characteristic equation $\displaystyle{\mathcal{P}_{a}(\lambda,\tau_{1},\tau_2)=0}$ has a real positive solution $\displaystyle{\lambda_a}$. This shows that $(T_{\ast,a},E_{\ast,a})$ is unstable.
\end{proof}
From now on, we will establish sufficient conditions for the local stability of this equilibrium. Let us begin with the following result.

\begin{theorem}
Let $a,\tau_1, \tau_2 \geq 0$ be small. 
\begin{itemize}
    \item[1.] If condition $\dagger)$ in Proposition \ref{Prop unique nontrivial equilibrium} holds, then $(T_{\ast,a}, E_{\ast,a})$ is a locally asymptotically stable equilibrium point of \eqref{Main EQ aut}.
\item[2.] If condition $\ddagger)$ in Proposition  \ref{Prop unique nontrivial equilibrium} holds, then $(T_{\ast}, E_{\ast})$ is locally asymptotically stable (stable node), provided that one of the following conditions holds:
\begin{itemize}
    \item[a.] $\displaystyle{\sigma\leq \frac{b \beta}{1+\beta}<b.}$
\item[b.] $\displaystyle{\frac{b \beta}{1+\beta}<\sigma<\sigma_{\Delta}}$ and $\displaystyle{0<\mu_{\Delta}(\sigma)<\mu< \mu_{c}}$ where
\[
\sigma_{\Delta}=\frac{b \beta}{1+\beta}\Big(1+\mu_{c}\Big(\frac{b}{1+\beta}\Big)^{1/\beta}\Big) \quad \text{and} \quad \mu_{\Delta}(\sigma)=\frac{1}{b\beta}\Big(\frac{1+\beta}{b}\Big)^{1/\beta}\big((1+\beta)\sigma-b\beta\big).
\]
\end{itemize}
\end{itemize}
\end{theorem}
\begin{proof}
For $\tau=(0,0)$, it is easy to see that, for all $a\geq 0$ we have
\[
\begin{split}
\mathcal{P}_{a}(\lambda,0,0)&=(\lambda-\lambda^{\ast}_{1,a})(\lambda-\lambda^{\ast}_{2,a})+\gamma(p-m) h^{\prime}\big(T_{\ast,a})T_{\ast,a} E_{\ast,a},\\
&=\lambda^2-(\lambda^{\ast}_{1,a}+\lambda^{\ast}_{2,a})\lambda+\mathscr{D}_{\ast,a}.
\end{split}
\]
where $\lambda^{\ast}_{1,a}, \lambda^{\ast}_{2,a}$ and $\mathscr{D}_{\ast,a}$ given by \eqref{valores de lambda} and \eqref{determinante} respectively. Moreover, 
\[
(\lambda_{1,a}+\lambda_{2,a})^2-4\mathscr{D}_{\ast,a}=(\lambda_{1,a}-\lambda_{2,a})^2+\frac{4\mu \eta T_{\ast,a}f(T_{\ast,a})}{(1+aT_{\ast,a})^{2}}.
\]
On the other hand, in particular, for $a=0$ we can rewrite equation \eqref{ecuacion para T unico}, as follows
\[
\mu(T^{\beta}_{\ast}-b)T_{\ast}+T^{\beta}_{\ast}=b-\sigma \quad \Leftrightarrow \quad (1+\mu T_{\ast})(b-T^{\beta}_{\ast})=\sigma,
\]
then $\displaystyle{1+\mu T_{\ast}>0}$ for all $\displaystyle{\mu \in \mathbb{R}}$. Now, from \eqref{Determinante cero} and the condition $\dagger)$ in Proposition \ref{Prop unique nontrivial equilibrium} we deduce that
\[
\lambda^{\ast}_{1,0}+\lambda^{\ast}_{2,0}<0 \quad \text{and} \quad \mathscr{D}_{\ast,0}>0.
\]
Consequently, $\mathcal{P}_{0}(\lambda,0,0)$ has two distinct roots, both with negative real part. This implies that $(T_{\ast}, E_{\ast})$ is locally asymptotically stable for small $\tau_1,\tau_2\geq 0$. The previous result and a simple argument of continuity complete the proof of part 1. of the theorem.

 \vspace{0.3 cm}
 For part 2a. Let us assume that condition $\ddagger)$ is met. Firstly, for all $\mu>0$ we have 
\[
\lambda_{1,a}+\lambda^{\ast}_{2,a}<0 \quad \text{and} \quad (\lambda_{1,a}+\lambda_{2,a})^2-4\mathscr{D}_{\ast,a}>0.
\]
Therefore, the roots of $\mathcal{P}_0(\lambda,0,0)$ are real.  Our attention now turns to the analysis of the positivity of  $\mathscr{D}_{\ast, a}$, which can be 
considered a function of $\displaystyle{\mathscr{D}_{\ast, a}(\mu)}$. Notice that
\begin{equation}\label{function Delta cero new}
\begin{split}
\mathscr{D}_{\ast,0}(\mu):&=-\eta T_{\ast}\big(\mu f(T_{\ast})+f^{\prime}(T_{\ast})\big(1+\mu T_{\ast}\big)\big),\\
\end{split}
\end{equation}
where $\displaystyle{T_{\ast}=T_{\ast}(\mu)}$ is the only solution of \eqref{ecuacion para T unico}. Recall that the positive sign of $\mathscr{D}_{\ast,0}(\mu)$ implies the local asymptotical stability of $(T_{\ast}, E_{\ast})$. With this in mind, we proceed as follows.
Since for all $0<\mu<\mu_c$ we have 
$$\displaystyle{T^{\beta}_{\ast}(0)=b-\sigma<T^{\beta}_{\ast}(\mu)<T^{\beta}_{\ast}(\mu_{c}),}$$
therefore,
\[
\text{If} \quad \frac{b}{1+\beta}\leq b-\sigma \quad \Leftrightarrow \quad \sigma \leq \frac{b\beta}{1+\beta}, \quad \text{then} \quad \frac{b}{1+\beta}<T_{\ast}^{\beta}.
\] 
As a consequence, $\mathscr{D}_{\ast,0}(\mu)>0,$ for all $\beta \in ]0,1[$ and $0<\mu<\mu_{c}$. Proof of part 2b: Notice that there exists $\mu_{\Delta}>0$ such that \eqref{ecuacion para T unico} has the solution $\displaystyle{T=\Big(\frac{b}{1+\beta}\Big)^{1/\beta}}$. 
In fact,
\[
\mu\Big(\frac{b}{1+\beta}-b\Big)\Big(\frac{b}{1+\beta}\Big)^{1/\beta}+\frac{b}{1+\beta}=b-\sigma,
\]
This implies,
\[
\mu=\mu_{\Delta}(\sigma):=\frac{1+\beta}{b\beta}\Big(\frac{1+\beta}{b}\Big)^{1/\beta}\Big(\sigma-\frac{b\beta}{1+\beta}\Big),
\]
which, in turn, implies $\displaystyle{\sigma>\frac{b \beta}{1+\beta}}$. At this point, we compare the value $\mu_{\Delta}(b)$ with respect to $\mu_{c}$. Then, straightforward computations show that
\[
\mu_{\Delta}(b)=\frac{1}{\beta}\Big(\frac{1+\beta}{b}\Big)^{1/\beta}>\mu_{c} \quad \Leftrightarrow \quad (1-\beta)^{\frac{1-\beta}{\beta}}>\frac{\beta}{(1+\beta)}.
\]
It is an easy exercise to prove that the last inequality is true for all $\beta \in \, [0,1].$ Therefore, we deduce the existence of a unique $\sigma=\sigma_{\Delta}$ such that $\displaystyle{\mu_{\Delta}(\sigma_{\Delta})=\mu_{c}}$ for all $\beta \in \, ]0,1[.$ Moreover, 
\[
\sigma_{\Delta}=\frac{b \beta}{1+\beta}\Big(1+\mu_{c}\Big(\frac{b}{1+\beta}\Big)^{1/\beta}\Big).
\]
Finally, following the proof of part 2a. for all $\displaystyle{\frac{b \beta}{1+\beta}<\sigma<\sigma_{\Delta}}$ and $\displaystyle{0<\mu_{\Delta}<\mu<\mu_c}$ we have
\[
T_{\ast}(\mu_{\Delta})=\Big(\frac{b}{1+\beta}\Big)^{1/\beta}<T_{\ast}(\mu),
\]
proving that $\mathscr{D}_{\ast,0}(\mu)>0.$ Taken $a\geq 0$ sufficiently small, the continuity of the function $\mathscr{D}_{\ast,a}(\mu)$ with respect $a$, complete the proof. 
\end{proof}


\begin{lemma} Assume $a,\tau_1. \tau_2 \geq 0$ to be small. If $\beta=1$ and $\displaystyle{\sigma<b}$, then $(T_{\ast.a}, E_{\ast.a})$ is locally asymptotically stable for all $\mu\in \mathbb{R}^{+}$. Moreover, if $\mu\geq 0$ is small, the same conclusion holds for all $\beta \in\, ]0,1[$ and $\sigma<b$.
\end{lemma}
\begin{proof}
Let $a=0$, $\tau=(0,0)$ and assume that $\sigma<b$. A direct computation shows that the function $\mathscr{D}_{\ast,0}(\mu)$ can be written as
\[
\begin{split}
\mathscr{D}_{\ast,0}(\mu)&=-\eta T_{\ast}\big(f^{\prime}(T_{\ast})+\mu(f(T_{\ast})+T_{\ast}f^{\prime}(T_{\ast}))\big)\\
\mathscr{D}_{\ast,0}(\mu)&=r\eta \big((\beta-1)T_{\ast}^{\beta}+b-\sigma+\mu \beta T_{\ast}^{\beta+1})).
\end{split}
\]
Then, for $\beta=1$ we have $\mathscr{D}_{\ast,a}(\mu)>0$ for all $\mu\geq 0$ and $a\geq 0$ sufficiently small.
Finally, if $\mu=0$ then $T_{\ast}^{\beta}=b-\sigma$, therefore
\[
\mathscr{D}_{\ast,0}(0)=r \eta  \big((\beta-1)(b-\sigma)+b-\sigma)\big)=r\beta\eta (b-\sigma)>0.
\]
A simple argument of continuity completes the proof.
\end{proof}

\subsection{Local stability for non-small delays}
This section is devoted to exploring the stability of $(T_{\ast,a},E_{\ast,a})$ for small $\displaystyle{a\geq 0}$ and $\tau_1,\tau_2\geq 0$ that are not necessarily small. For simplicity in the following computations, we rewrite the characteristic function \eqref{eq-caracteristica} as follows
\[
\mathcal{P}_{a}(\lambda,\tau_1,\tau_2)=P(\lambda)+R(pe^{-\tau_1\lambda}-me^{-\tau_2\lambda}),
\]
with $\displaystyle{P(\lambda)=(\lambda-\lambda^{\ast}_{1,a})(\lambda-\lambda^{\ast}_{2,a})}$ and $\displaystyle{R:=f(T_{\ast,a})h^{\prime}(T_{\ast,a})T_{\ast,a}}$. Let us assume that $\lambda=x+iy$, $x,y\in \mathbb{R}$ is a solution of \eqref{eq-caracteristica}. Dividing the equation into the real and imaginary parts,  we obtain 
\begin{equation}\label{newsystem for lambda full}
\begin{split}
P(x)-y^2+R\big(p\cos(\tau_{1}y)e^{-\tau_{1}x}-m\cos(\tau_{2}y)e^{-\tau_{2}x}\big)&=0,\\
2xy-(\lambda^{\ast}_{1,a}+\lambda^{\ast}_{2,a})y-R\big(p\sin(\tau_{1}y)e^{-\tau_{1}x}-m\sin(\tau_{2}y)e^{-\tau_{2}x}\big)&=0.
\end{split}
\end{equation}
Notice that if $\tau_1=\tau_2=\tau$ we obtain
\begin{equation}\label{characteristic function with equal delays}
\mathcal{P}_{a}(\lambda,\tau)=P(\lambda)-Ne^{-\tau\lambda},
\end{equation}
with $N:=(m-p)R$ which can be divided into \footnote{In the set of scale parameters \eqref{parametros-escalados} we have $\displaystyle{N=\frac{\eta \mu T_{\ast,a} f(T_{\ast,a})}{(1+aT_{\ast,a})^2}}$. Further, for $a=0$, $N=\eta \mu T_{\ast}f(T_{\ast})$.} 
\begin{equation}\label{newsystem for lambda full equals}
\begin{split}
P(x)-y^2-N\cos(\tau y)e^{-\tau x}&=0,\\
2xy-(\lambda^{\ast}_{1,a}+\lambda^{\ast}_{2,a})y+N\sin(\tau y)e^{-\tau x}&=0.
\end{split}
\end{equation}
We begin our analysis considering the characteristic function \eqref{characteristic function with equal delays}. 
To this end, we apply Theorems A, B, and C in the Appendix. 

\begin{theorem} Assume that $a\geq0$ is sufficiently small. If condition $\dagger)$ in Proposition \ref{Prop unique nontrivial equilibrium} holds and 
\[
f^{\prime}(T_{\ast})(\mu T_{\ast}+1)-\mu f(T_{\ast})<0,
\]    
then $(T_{\ast,a}, E_{\ast,a})$ is a locally asymptotically stable equilibrium point of \eqref{Main EQ aut} for all values of $\tau$ satisfying
    \[
    0\leq \tau< \frac{T_{\ast,a}f^{\prime}(T_{\ast,a})-\eta\Big(\frac{\mu T_{\ast}}{1+aT_{\ast,a}}+1\Big)}{\frac{\mu \eta T_{\ast,a}f(T_{\ast,a})}{(1+a T_{\ast,a})^2}} \approx  \frac{T_{\ast}f^{\prime}(T_{\ast})-\eta(\mu T_{\ast}+1)}{\mu \eta T_{\ast}f(T_{\ast})}.
    \]
Moreover, for $a=0$, there exists $\tau_{c}>0$ given by
\[
\tau_{c}=\frac{1}{\hat{y}}\arctan\Big[-\frac{(\lambda^{\ast}_{1,0}+\lambda^{\ast}_{2,0})\hat{y}}{\lambda^{\ast}_{1,0}\lambda^{\ast}_{2,0}-\hat{y}^2}\Big]
\]
with $\hat{y}$ the smallest positive solution of
\[
G(y)=y^{4}+\Big(\big(\lambda^{\ast}_{1,0})^{2}+\big(\lambda^{\ast}_{2,0}\big)^{2}\Big)y^2+\big(\lambda^{\ast}_{1.0}\lambda^{\ast}_{2,0})^{2}-N^{2}, 
\]
with $N=\eta \mu T_{\ast} f(T_{\ast})$ for which $\tau< \tau_0$ the equilibrium point $(T_{\ast},E_{\ast})$ is asymptotically stable and for $\tau< \tau_0$ is unstable. Furthermore, as  $\tau$ increases through $\tau_0$, $(T_{\ast}, E_{\ast})$ bifurcates into stable periodic solutions with small amplitude.
\end{theorem}

\begin{proof}
Let us begin with the proof of part 1. Consider $\lambda^{\ast}_{1,a}$ and $\lambda^{\ast}_{2,a}$ to be given by \eqref{valores de lambda} namely
\[
\lambda^{\ast}_{1,a}=T_{\ast,a}f^{\prime}(T_{\ast,a})<0 \quad \text{and } \quad \lambda^{\ast}_{2,a}=-\eta\Big(\frac{\mu T_{\ast,a}}{1+aT_{\ast,a}}+1\Big).
\]
Recall that $\lambda^{\ast}_{2,0}<0$ for all $\mu \in \mathbb{R}$. Define the function $\displaystyle{\mathcal{K}(a):=\lambda^{\ast}_{1,a}\lambda^{\ast}_{2,a}+N}$, explicitly given by 
\[
\mathcal{K}(a):=-\eta T_{\ast,a}\Big[f^{\prime}(T_{\ast,a})\Big(\frac{\mu T_{\ast,a}}{1+aT_{\ast,a}}+1\Big)- \Big(\frac{\mu f(T_{\ast,a})}{(1+aT_{\ast,a})^2}\Big)\Big],
\]
with 
\[
\mathcal{K}(0)=-\eta T_{\ast}\big(f^{\prime}(T_{\ast})(\mu T_{\ast}+1)-\mu f(T_{\ast})\big).
\]
By hypothesis $\mathcal{K}(0)>0$ we deduce the following:
\begin{itemize}
\item [$1.$] $\lambda^{\ast}_{2,a}<0$ and $\lambda^{\ast}_{1,a}+\lambda^{\ast}_{2,a}<0,$
    \item [$2.$] $\mathcal{K}(a)=\lambda^{\ast}_{1,a}\lambda^{\ast}_{2,a}+N<0 \quad \Leftrightarrow \quad \lambda^{\ast}_{1,a}\lambda^{\ast}_{2,a}<-N, $ 
\end{itemize}
for all $a\geq 0$ and sufficiently small. By continuity and Theorem A., the equilibrium point $(T_{\ast,a},E_{\ast,a})$ is asymptotically stable for all $\tau \in [0,\tau_{a})$ with  
\[
\tau_{a}=\frac{\lambda^{\ast}_{1,a}+\lambda^{\ast}_{2,a}}{N}=\frac{T_{\ast,a}f^{\prime}(T_{\ast,a})-\eta\Big(\frac{\mu T_{\ast}}{1+aT_{\ast,a}}+1\Big)}{\frac{\mu \eta T_{\ast,a}f(T_{\ast,a})}{(1+a T_{\ast,a})^2}},
\]
for all $a\geq 0$ and sufficiently small, where $\displaystyle{\tau_{a}\approx \tau_{0}:=\frac{T_{\ast}f^{\prime}(T_{\ast})-\eta(\mu T_{\ast}+1)}{\mu \eta T_{\ast}f(T_{\ast})}.}$

Now, for the calculation of $\tau_{c}$ for which \eqref{newsystem for lambda full equals} has a solution of the form $\lambda=iy$ (i.e., $x=0$) we have (for $a=0$)
\begin{equation}\label{ecuación para tau crítico}
\begin{split}
\lambda^{\ast}_{1,0}\lambda^{\ast}_{2,0}-y^2-N\cos(\tau y)&=0,\\
-(\lambda^{\ast}_{1,0}+\lambda^{\ast}_{2,0})y+N\sin(\tau y)&=0,
\end{split}
\end{equation}
implying that $G(y)=0$ with $G$ the quadratic function in $y^{2}$ given by
\[
G(y):=y^{4}+\Big(\big(\lambda^{\ast}_{1,0})^{2}+\big(\lambda^{\ast}_{2,0}\big)^{2}\Big)y^2+(\lambda^{\ast}_{1,0}\lambda^{\ast}_{2,0}-N)(\lambda^{\ast}_{1,0}\lambda^{\ast}_{2,0}+N).
\]
where $N=\eta \mu T_{\ast}f(T_{\ast})$. Since $\displaystyle{G(0)=(\lambda^{\ast}_{1,0}\lambda^{\ast}_{2,0}-N)(\lambda^{\ast}_{1,0}\lambda^{\ast}_{2,0}+N)}<0$ there exist $\hat{y}>0$ such that $\displaystyle{G(\hat{y})=0}$ moreover, $\displaystyle{\hat{y}}$ is given explicitly by
\[
\hat{y}=\sqrt{\frac{-\Big[\big(\lambda^{\ast}_{1,0})^{2}+\big(\lambda^{\ast}_{2,0}\big)^{2}\Big]+\sqrt{\Big[\big(\lambda^{\ast}_{1,0})^{2}+\big(\lambda^{\ast}_{2,0}\big)^{2}\Big]^2-4\Big(\big(\lambda^{\ast}_{1.0}\lambda^{\ast}_{2,0})^{2}-N^{2}\Big)}}{2}}
\]
From \eqref{ecuación para tau crítico} we deduce
\[
\tan (\tau y)=-\frac{(\lambda^{\ast}_{1,0}+\lambda^{\ast}_{2,0})y}{\lambda^{\ast}_{1,0}\lambda^{\ast}_{2,0}-y^2}
\]
The critical value $\tau_{c}$ is given by
\[
\tau_{c}:=\tau_{0} \quad \text{where} \quad \tau_{k}:=\frac{1}{\hat{y}}\arctan\Big[-\frac{(\lambda^{\ast}_{1,0}+\lambda^{\ast}_{2,0})\hat{y}}{\lambda^{\ast}_{1,0}\lambda^{\ast}_{2,0}-\hat{y}^2}\Big]+\frac{2k\pi}{\hat{y}}, \quad k=0,1,2, \dots
\]
The conclusion of the existence of a Hopf bifurcation at the critical time-delay value $\tau_0$  follows from Theorem C. This completes the proof.
\end{proof}

\textbf{Stability for non equal time delays.} The next lines provides the general frame where the stability of the equilibrium point $(T_{\ast,a},E_{\ast,a})$ is guaranteed for $\tau_1,\tau_2\in \mathbb{R}^2_{+}$. From system \eqref{newsystem for lambda full} we have
\[
\begin{split}
\mathcal{Q}_{0}(x,y)+Rp \cos(\tau_{1}y)e^{-\tau_{1}x}=Rm\cos(\tau_{2}y)e^{-\tau_{2}x},\\
\mathcal{Q}_{1}(x,y)-Rp \sin(\tau_{1}y)e^{-\tau_{1}x}=-Rm\sin(\tau_{2}y)e^{-\tau_{2}x},
\end{split}
\]
where
\[
\mathcal{Q}(x,y)=P(x)-y^2 \quad \text{and} \quad \mathcal{Q}_{1}(x,y)=2xy-(\lambda^{\ast}_{1,a}+\lambda^{\ast}_{2,a})y.
\]
A direct calculation shows that
\[
\mathcal{Q}^2_{0}(x,y)+\mathcal{Q}^2_{1}(x,y)+2Rp( \mathcal{Q}_{0}(x,y)\cos(\tau_{1}y)- \mathcal{Q}_{1}(x,y)\sin(\tau_{1}y))e^{-\tau_{1}x}=(Rm)^2e^{-2\tau_{2}x}-(Rp)^{2}e^{-2\tau_{1}x}.
\]

Now, we look for solutions of the form $\lambda=iy$ (i.e., $x=0$). In this case, we obtain the following
\begin{equation}\label{stability switching curves}
\mathcal{Q}^2_{0}(y)+\mathcal{Q}^2_{1}(y)+2Rp(\mathcal{Q}_{0}(y)\cos(\tau_1 y)-\mathcal{Q}_{1}(y)\sin(\tau_1 y))=(Rm)^2-(Rp)^2,
\end{equation}
with
\[
\mathcal{Q}_{0}(y):=\mathcal{Q}_{0}(0,y)=\lambda^{\ast}_{1,a}\lambda^{\ast}_{2,a}-y^2 \quad \text{and} \quad \mathcal{Q}_{1}(y):=\mathcal{Q}_{1}(0,y)=-(\lambda^{\ast}_{1,a}+\lambda^{\ast}_{2,a})y.
\]
Under the condition $\dagger)$ in Proposition \ref{Prop unique nontrivial equilibrium}, for all $a\geq 0$ sufficiently small  we have $\lambda^{\ast}_{1,a}+\lambda^{\ast}_{2,a}<0$ which implies $\displaystyle{\mathcal{H}^2(y)=\mathcal{Q}^2_{0}(y)+\mathcal{Q}^2_{1}(y)>0}$. Consequently, there exists a continuous function $\varphi_{1}: \mathbb{R} \to \mathbb{R}$, $y \to \varphi_1(y)$ such that
\[
\mathcal{Q}_{0}(y)=\sqrt{\mathcal{Q}^2_{0}(y)+\mathcal{Q}^2_{1}(y)}\cos(\varphi_1(y)) \quad \text{and} \quad \mathcal{Q}_{1}(y)=\sqrt{\mathcal{Q}^2_{0}(y)+\mathcal{Q}^2_{1}(y)}\sin(\varphi_1(y)).
\]
In this scenario, equation \eqref{stability switching curves} becomes
\begin{equation}\label{ecuacion para el primer delay}
(\lambda^{\ast}_{1,a}\lambda^{\ast}_{2,a}-y^2)^2+(\lambda^{\ast}_{1,a}+\lambda^{\ast}_{2,a})^2y^2+R^2(p^2-m^2)=-2Rp \mathcal{H}(y)\cos(\tau_1y+\varphi_1(y)).
\end{equation}
There exists $\tau_1 \in \mathbb{R}_{+}$ satisfying the last equation if and only if
\begin{equation}\label{fundamental inequality for delay one}
\big|(\lambda^{\ast}_{1,a}\lambda^{\ast}_{2,a}-y^2)^2+(\lambda^{\ast}_{1,a}+\lambda^{\ast}_{2,a})^2y^2+R^2(p^2-m^2)\big|\leq 2Rp\mathcal{H}(y).
\end{equation}
Define $I\in \mathbb{R}_{+}$ so that condition \eqref{fundamental inequality for delay one} is met. For each $\hat{y} \in I$ define the function
\[
\vartheta(\hat{y}):=\arccos\Big(-\frac{(\lambda^{\ast}_{1,a}\lambda^{\ast}_{2,a}-\hat{y}^2)^2+(\lambda^{\ast}_{1,a}+\lambda^{\ast}_{2,a})^2\hat{y}^2+R^2(p^2-m^2)}{2Rp\sqrt{(\lambda^{\ast}_{1,a}\lambda^{\ast}_{2,a}-\hat{y}^2)^2+(\lambda^{\ast}_{1,a}+\lambda^{\ast}_{2,a})^2\hat{y}^2})}\Big), \quad \vartheta \in [0,\pi]
\]
Finally, from \eqref{ecuacion para el primer delay} we have
\begin{equation}\label{valor de tau1}
\tau^{\pm}_{1,s}(\hat{y}):=\frac{\pm \vartheta(\hat{y})-\varphi_1(\hat{y})}{\hat{y}}+\frac{2s\pi}{\hat{y}}, \quad s\in \mathbb{Z}.
\end{equation}
From this, it is possible to obtain an explicit formula for $\tau_2=\tau_{2}(\hat{y})$ throughout the system \eqref{newsystem for lambda full} as follows
\[
\tan (\tau_2 \hat{y})=\frac{Rp \sin(\tau_{1}(\hat{y})\hat{y})-\mathcal{Q}_1(\hat{y})}{\mathcal{Q}_{0}(\hat{y})+Rp \cos(\tau_{1}(\hat{y})\hat{y})},
\]
then,
\begin{equation}\label{valor de tau2}
\tau^{\pm}_{2,k}(\hat{y}):=\frac{1}{\hat{y}}\arctan\Big[\frac{Rp \sin(\tau^{\pm}_{1}(\hat{y})\hat{y})-\mathcal{Q}_1(\hat{y})}{\mathcal{Q}_{0}(\hat{y})+Rp \cos(\tau^{\pm}_{1}(\hat{y})\hat{y})}\Big]+\frac{2k\pi}{\hat{y}}, \quad k\in \mathbb{Z}.
\end{equation}

The previous computation provides the proof of the following result.
\begin{theorem}(Stability crossing curves)
Assume that $a\geq 0$ is sufficiently small and that the condition $\dagger)$ in Proposition \ref{Prop unique nontrivial equilibrium}. Let $I\in \mathbb{R}_{+}$ be the set where \eqref{fundamental inequality for delay one} holds. The stability of the equilibrium point $(T_{\ast,a},E_{\ast,a})$ changes if $(\tau_1,\tau_2) \in \mathbb{R}_{+}$ crosses the curves
\[
\mathcal{C}_{s,k}=\left\{(\tau^{\pm}_{1}(y),\tau^{\pm}(y)) \in \mathbb{R}^{2}_{+}:y\in I, s,k \in \mathbb{Z}\right\}.
\]
with $\displaystyle{\tau^{\pm}_{1},\tau^{\pm}_{2}}$ given by \eqref{valor de tau1} y \eqref{valor de tau2} respectively.
\end{theorem}

\section{Continuation}\label{continuacion}

In this 
section, 
we shall prove the existence of $\omega$-periodic solutions for small perturbations of the autonomous undelayed system. 
Keeping in mind Theorem \ref{thm-impl} consider any equilibrium point $(\hat{T},\hat{E})$  of (\ref{main eq 3}) for $\varphi=0$ and $\tau_1=\tau_2=0$  and compute the
partial derivative of $\Phi$ at $(\hat{T},\hat{E},0)$. 
which is given by 
$$D_{(T,E)}\Phi(\hat{T},\hat{E},0)Z=  \lim_{h \to 0}\frac{\Phi((\hat{T},\hat{E}+hZ,0) -\Phi((\hat{T},\hat{E},0)}{h}, 
$$
that is, 
$$D_{(T,E)}\Phi(T^*,E^*,0)Z= MZ,
$$
where $M$ is the matrix given by
\[
M:=\begin{pmatrix}
f(\hat{T})-\hat{T}f^{\prime}(\hat{T})-\gamma \hat{E} & -\gamma \hat{T} \\
(p-m)\hat{E}h^{\prime}(\hat{T})& (p-m)h(\hat{T})-\eta
\end{pmatrix}.
\]
Let us first assume that $(\hat{T},\hat{E})$ is the tumor-free equilibrium point, that is, $\displaystyle{(\hat{T},\hat{E})=(0,\sigma/\eta)}$. Then 
\[
M_{|_{(0,\sigma/\eta)}}:=\begin{pmatrix}
rb-\gamma \sigma/\eta & 0\\
\sigma(p-m) h^{\prime}(0)/\eta & -\eta
\end{pmatrix}.
\]

To conclude our analysis, 
let us proceed with an strictly positive equilibrium $(\hat{T},\hat{E})$ given as a solution of the system \eqref{nontrivial equilibrium systems}, that is,
\[
f(\hat{T})=\gamma \hat{E}, \quad \text{and} \quad \sigma=\hat{E}\big(\eta+(m-p)h(\hat{T})\big),
\]
with $\hat{T},\hat{E}>0$. In particular, for $a=0$, from Proposition \ref{ecuacion para T unico} we deduce that $\hat{T}=T_{\ast}$ under some appropriate conditions. In addition, the corresponding matrix is 
\[
M_{|_{(\hat{T},\hat{E})}}:=\begin{pmatrix}
\hat{T}f^{\prime}(\hat{T}) & -\gamma \hat{T}\\
(p-m)\hat{E}h^{\prime}(\hat{T}) & -\sigma/\hat{E}
\end{pmatrix}
\]
Thus, taking into account Lemma \ref{iso}, existence of $\omega$-periodic solution  for all values of $\omega$ is deduced in both cases. 
\begin{theorem}  
Let $(T^*,E^*)$ be an admissible equilibrium of (\ref{main eq 3}) with 
$b\equiv b_0$. 
Then there exist
constants $\varphi^*,\tau^*_j>0$ and $r^*>0$ such that the
system (\ref{main eq 3}) has at least one $\omega$-periodic 
solution $(T,E)$ with $\|(T-T^*,E-E^*)\|_\infty <r^*$, provided that
$\|\varphi\|_\infty< \varphi^*, \tau_1<\tau_1^*$ and $\tau_2<\tau_2^*$. 
\end{theorem}
\begin{proof}
From the previous calculations, when $T^*=0$, it is verified
that the matrix $M_{|_{(0,\sigma/\eta)}}$ has only real
eigenvalues given by
\[
\lambda_{1}=rb-\gamma \sigma/\eta, \quad \text{and} \quad \lambda_2=-\eta.
\]
On the other hand, for $T^*>0$   
it is seen that the trace of the matrix $M_{|_{(T^{\ast},E^{\ast})}}$ is strictly negative, indeed,
\[
tr ( M_{|_{(T^{\ast},E^{\ast})}})=-r\beta [T^{\ast}]^{\beta}-\sigma/E^{\ast}<0,
\]
which, in particular,  implies that it has no
pure imaginary eigenvalues. Thus, in both cases, the conclusion follows from lemma \ref{iso}.  
\end{proof}

\section{Numerical Validation}\label{NV}
Section in construction.

\section*{Appendix A}
In Section \ref{section Basic results} for each $\beta \in \, ]0,1]$ we consider the function
\[
h_{\mu}(T):=S(T,\mu,0)=\mu(T^{\beta}-b)T+T^{\beta}, 
\]
with $\displaystyle{0\leq T<b^{1/\beta}}$ and $\mu \in \mathbb{R}$. This function provides the first component of non-trivial and admissible  equilibrium points $(T, E)$ (with $T, E>0$) of the system \eqref{Main EQ aut} as solutions of the algebraic equation 
\[
h_{\mu}(T)=h_{0}. \qquad (\blacktriangle)
\]
with $h_{0}\in \mathbb{R}$. In the following lines, we present mathematically and numerically some algebraic properties of $h_{\mu}$ which are important to study the number of solutions of $(\blacktriangle)$. Firstly, notice that
\[
\begin{split}
h_{\mu}^{\prime}(T)&=T^{\beta-1}(\beta+\mu(1+\beta)T)-\mu b,\\
h_{\mu}^{\prime \prime}(T)&=\beta T^{\beta-2}(\beta-1+\mu(1+\beta)T),\\
h_{\mu}^{\prime \prime \prime}(T)&=\beta (\beta-1)T^{\beta-3}(\beta-2+\mu(1+\beta)T).
\end{split}
\]
In addition,
\[
h_{\mu}(b^{1/\beta})=b, \quad  \quad \text{and}
\quad \quad  h_{\mu}^{\prime}(b^{1/\beta})=\beta b\big(b^{-1/\beta}+\mu\big).
\]
Let us start our analysis by assuming the case $\mu<0$. Straightforward computations show the following:
\begin{itemize}
\item[$\dagger)$] $h_{\mu}^{\prime \prime}(T)<0$ for all $T>0$. Then, $h_{\mu}(T)$ is a smooth concave function, and satisfies
\[
\lim_{T\to 0^{+}}h^{\prime}_{\mu}(T)=\begin{cases}
0, \quad \text{if} \quad \beta=1,\\
\infty, \quad \text{if} \quad 0<\beta<1,
\end{cases} \quad \text{and} \quad \lim_{T\to \infty}h^{\prime}_{\mu}(T)=-\infty.
\]
In consequence, there exists exactly one (positive) critical point $T_{\Delta}$ given implicitly by the equation
\begin{equation}\label{ecuacion pc}
T_{\Delta}^{\beta-1}(\beta+\mu(1+\beta)T_{\Delta})=\mu b.
\end{equation}
\item[$\dagger)$] $h_{\mu}(T)>0$ for all $0<T<b^{1/\beta}$. Furthermore, from \eqref{ecuacion pc} we deduce that
\[
h_{\mu}(T_{\Delta})=T_{\Delta}^{\beta}\big( 1-\beta-\mu \beta T_{\Delta}\big).
\]
\end{itemize}
Finally, it is clear that $\displaystyle{b^{1/\beta}\lesseqgtr T_{\Delta}}$ if $\displaystyle{b^{-1/\beta}\gtreqless -\mu}$. 
The following Lemma is easily deduced from the properties of $h_{\mu}(T)$.

\begin{lemma} \label{lema function h for mu negativo} Let us assume $\mu<0$. Then
 \begin{itemize}
 \item [1.] If $h_{0}<0$, then there are no solutions of $(\blacktriangle)$ in $[0,b^{1/\beta}]$.
     \item [2.] 
     \begin{itemize}
 \item[a.] If $0\leq  h_{0}\leq b$, then there is exactly one solution of $(\blacktriangle)$ in $[0,b^{1/\beta}]$. 
 \item[b.] If $b<h_{0}$ and satisfies
     \[
     h_{0}\leq T_{\Delta}^{\beta}\big( 1-\beta-\mu \beta T_{\Delta}\big),
     \]
     with $T_{\Delta}$ the only positive solution of \eqref{ecuacion pc}. Then, there are at most two solutions of $(\blacktriangle)$ in $[0,b^{1/\beta}]$ if $b^{1/\beta} \geq T_{\Delta}$ and there are no solutions in $[0,b^{1/\beta}]$  if  $b^{1/\beta} <T_{\Delta}$.
     \end{itemize}
 \end{itemize}
\end{lemma}
\begin{remark}
Since for all $\mu<0$, 
the function $h$ is concave, namely $h^{\prime \prime}_{\mu}(T)<0$ for all $T>0$, it follows that
\[
\mu\leq -b^{-1/\beta} \quad \Leftrightarrow \quad b^{1/\beta}\geq T_{\Delta}, \quad \text{and} \quad  -b^{-1/\beta}<\mu<0 \quad \Leftrightarrow \quad b^{1/\beta}< T_{\Delta}.
\]
\end{remark}

\begin{figure}[t]
\centering
\begin{overpic}[width = 0.55\textwidth, tics = 6]{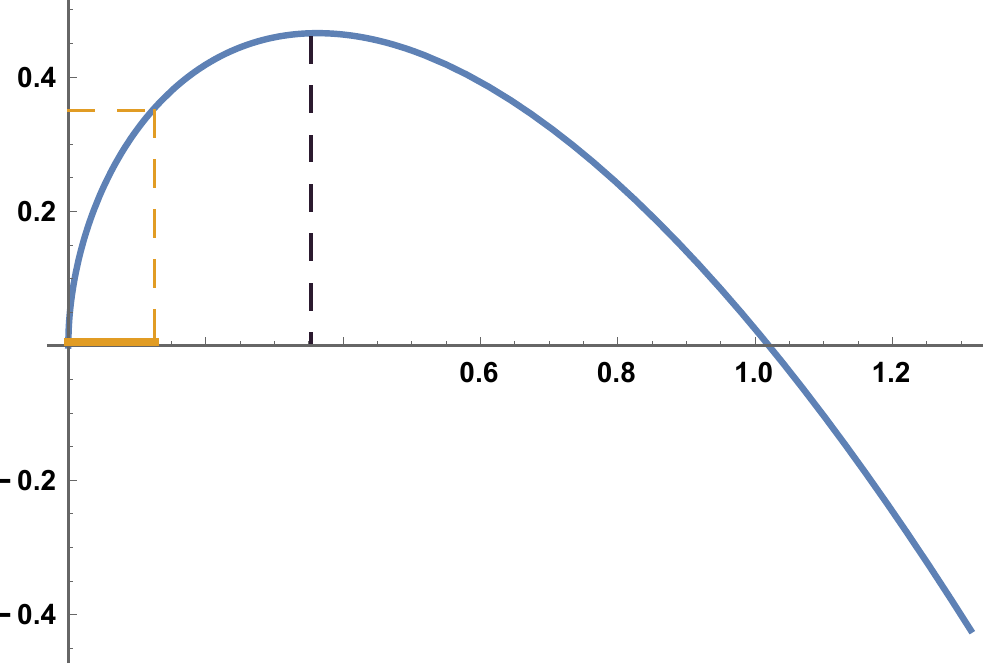}
\put(95.4, 10.5){$h_{\mu}(T)$}
\put(99.5, 27.5){$T$}
\put(29.5, 27.5){$T_{\Delta}$}
\put(13.5, 27.5){$b^{1/\beta}$}
\put(3.5, 54.5){$b$}
\end{overpic}
\caption{Graph of $h_{\mu}(T)$ with parameters $\mu=-1.5$, $b=0.35$, and $\beta=0.5$. For these values, $h_{\mu}^{\prime}(b^{1/\beta}) \approx 1.17$, therefore $b^{1/\beta}<T_{\Delta}$. If $0<h_{0}\leq b$ there exists exactly one solution of $(\blacktriangle)$ on $]0,b^{1/\beta}]$ and no solutions if $h_{0}>b.$}
\label{fig:Fig h(T)-mu-negativa}
\end{figure}
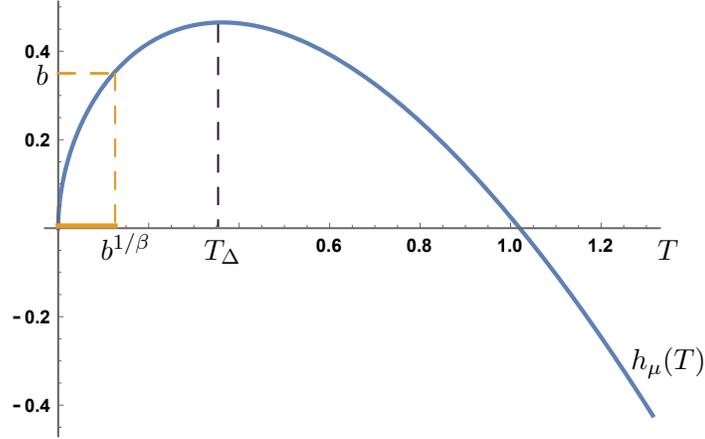

Next, we proceed with the case $\mu>0$. Firstly, note that
\[
\lim_{T\to 0^{+}}h^{\prime}_{\mu}(T)=\begin{cases}
0, \quad \text{if} \quad \beta=1,\\
\infty, \quad \text{if} \quad 0<\beta<1,
\end{cases}  \quad \text{and} \quad \lim_{T\to \infty}h^{\prime}_{\mu}(T)=\infty.
\]
In addition, $h_{\mu}^{\prime}(T)$ admits exactly one positive critical point $T_{\star}$. Indeed,
\[
h_{\mu}^{\prime \prime}(T_{\star})=0 \quad \Leftrightarrow \quad \beta-1+\mu(1+\beta)T_{\star}=0, \quad \Leftrightarrow \quad T_{\star}=\frac{1-\beta}{\mu(1+\beta)}.
\]
Moreover,
\[
h_{\mu}^{\prime \prime \prime}(T_{\star})=\beta(1-\beta) T_{\star}^{\beta-3}>0, \quad \text{and} \quad h_{\mu}^{\prime}(T_{\star})=\Big(\frac{1-\beta}{\mu (1+\beta)}\Big)^{\beta-1}-\mu b.
\]
The last left-hand inequality implies that $T_{\star}$ is a global minimum of $h_{\mu}^{\prime}(T)$ on $\left\{T\in \mathbb{R}:T\geq 0\right\}$. 
These computations provide the elements for the proof of the following result

\begin{figure}[h]
\centering
\begin{overpic}[width = 0.55\textwidth, tics = 6]{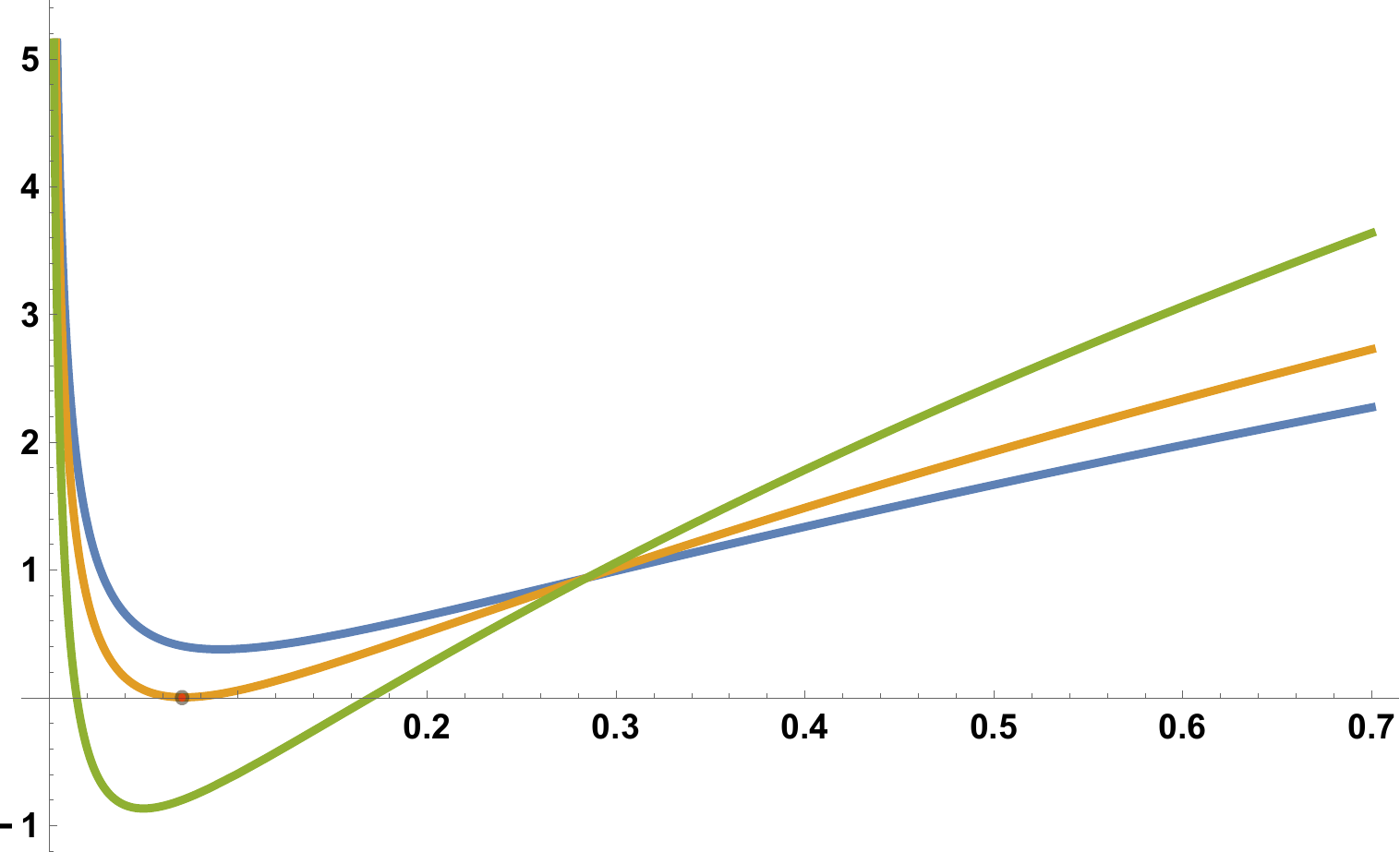}
\put(101.4, 7.5){$T$}
\put(99.5, 28.5){\small{$h^{\prime}_{\mu_{0}}(T)$}}
\put(99.5, 35.7){\small{$h^{\prime}_{\mu_{1}}(T)$}}
\put(99.5, 44.5){\small{$h^{\prime}_{\mu_{2}}(T)$}}
\put(8.5, 7.5){\tiny{$T_{\star}(\mu_{1})$}}
\end{overpic}
\caption{Graph of $h^{\prime}_{\mu}(T)$ with parameters $b=0.8$, and $\beta=0.5$ for $T\in ]0,0.7]$. Three different values of $\mu$ are considered: $\mu_{n}=3.68+n$, $n \in \left\{0,1,2\right\}.$ Notice that $\mu_{1}\approx \mu_{c}=75/16$. For $\mu<\mu_c$, $h^{\prime}_{\mu}(T)>0$ for all $T>0$. Instead, if $\mu>\mu_c$, $h^{\prime}_{\mu}(T)$ have exactly two positive zeros.}
\label{fig:Fig h(T)-mu-positiva-derivada}
\end{figure}
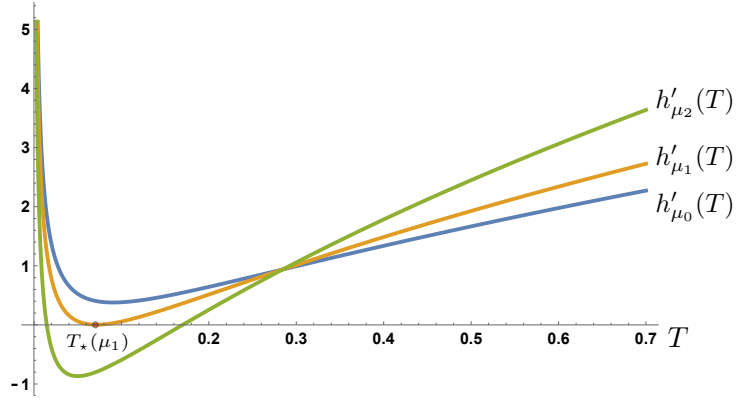

\begin{lemma} \label{lema function h for mu positive 1}
Assume that
\[
0<\mu\leq \mu_{c}:=\Big(\frac{1}{b}\Big(\frac{1-\beta}{1+\beta}\Big)^{\beta-1}\Big)^{\frac{1}{\beta}}.
\]
Then, 
\begin{itemize}
    \item[a.] If $h_{0}<0 $, there are no solutions of $(\blacktriangle)$ for all $T\geq 0$.
    \item[b.]  If $0\leq h_{0}$ there exist exactly one solution of $(\blacktriangle)$ on the set $T\geq 0$. In particular, if $h_0\leq b $ the solution lying on $[0,b^{1/\beta}].$
\end{itemize}
\end{lemma}
\begin{proof}
For each $\mu>0$, the previous computations shows that $h_{\mu}^{\prime}(T)$ has a global minimum at $\displaystyle{T=T_{\star}=\frac{1-\beta}{\mu(1+\beta)}}$. Furthermore,
\[
h_{\mu}^{\prime}(T_{\star})=\Big(\frac{1-\beta}{\mu (1+\beta)}\Big)^{\beta-1}-\mu b \geq 0 \quad \text{if and only if} \quad \mu\leq \Big(\frac{1}{b}\Big(\frac{1-\beta}{1+\beta}\Big)^{\beta-1}\Big)^{\frac{1}{\beta}},
\]
hence, $h_{\mu}(T)$ is a monotone-increasing continuous function. Since $h_{\mu}(0)=0,$  we have that $h_{\mu}(T)$ is non-negative, proving already statement a. and the first part of statement b.  Finally, since $h_{\mu}([0,b^{1/\beta}])\subseteq [0,b]$ clearly the solution of $(\blacktriangle)$ for $0\leq h_0\leq b$ lying on $[0,b^{1/\beta}]$.
\end{proof}

Let us consider the function $\displaystyle{l: ]0,\infty[\to \mathbb{R}}$, $\displaystyle{\mu \to l(\mu):=\min_{T>0} h_{\mu}^{\prime}(T)}$. Since $\displaystyle{T_{\star}=T_{\star}(\mu)}$ is the global minimum of $h^{\prime}_{\mu}(T)$ on  $\left\{T\in \mathbb{R}:T\geq 0\right\}$, it follows
\[
\begin{split}
l(\mu)=h_{\mu}^{\prime}(T_{\star}(\mu))&=T^{\beta-1}_{\star}(\mu)-\mu b=\mu^{1-\beta}\Big(\frac{1-\beta}{1+\beta}\Big)^{\beta-1}-\mu b,\\
&=\mu b\Big(\Big(\frac{\mu_c}{\mu}\Big)^{\beta}-1\Big).
\end{split}
\]
In consequence, $\displaystyle{l(\mu)}$ is monotone decreasing for all $\mu>0$ satisfying
\[
l(\mu_{c})=0, \quad \text{and} \quad  \lim_{\mu \to \infty}l(\mu)=-\infty.
\]
It should be pointed out that $h_{\mu}^{\prime \prime}(T)$ has only one sign change (precisely at $T_{\star}$). In fact,
\[
h_{\mu}^{\prime \prime \prime}(T_{\star})>0, \quad  \lim_{T \to 0}h_{\mu}^{\prime \prime}(T)=-\infty, \quad \text{and} \quad \lim_{T \to \infty}h_{\mu}^{\prime \prime}(T(\mu))=0.
\]

\begin{figure}[h]
\centering
\begin{overpic}[width = 0.95\textwidth, tics = 6]{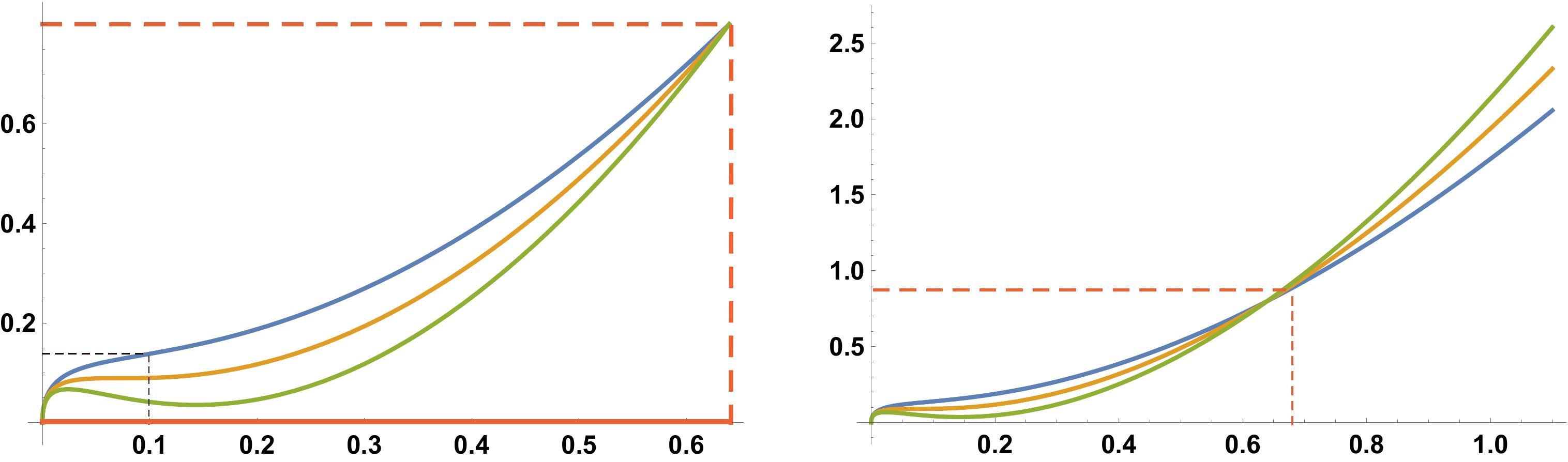}
\put(100.4, 0.5){$T$}
\put(46.5, 0.5){\small{$b^{1/\beta}$}}
\put(99.5, 28.5){\small{$h_{\mu_{2}}(T)$}}
\put(99.5, 24.7){\small{$h_{\mu_{1}}(T)$}}
\put(99.5, 21.5){\small{$h_{\mu_{0}}(T)$}}
\put(1, 27.5){\small{$b$}}
\put(-0.4, 6.5){\tiny{$h_{0}$}}
\put(6.8, 7.9){\tiny{$h_{\mu_{0}}(T)$}}
\put(13.5, 5.5){\tiny{$h_{\mu_{1}}(T)$}}
\put(20.5, 3.5){\tiny{$h_{\mu_{2}}(T)$}}
\put(81.5, 0.1){\tiny{$b^{1/\beta}$}}
\put(54, 10.1){\tiny{$b$}}
\end{overpic}
\caption{Graph of $h_{\mu}(T)$ with parameters $b=0.8$, and $\beta=0.5$. (Left) $T\in \, ]0,b^{1/\beta}]$. (Right) $T\in \,]0,1.2]$. Three different values of $\mu$ are considered: $\mu_{n}=3.68+n$, $n \in \left\{0,1,2\right\}.$ Notice that $\mu_{1}\approx \mu_{c}=75/16$. For $\mu\leq \mu_c$, there is exactly one solution of ($\blacktriangle$) if $0 \leq h_{0}\leq b$ lying on $[0,b^{1/\beta}]$.}
\label{fig:Fig h(T)-mu-positiva}
\end{figure}
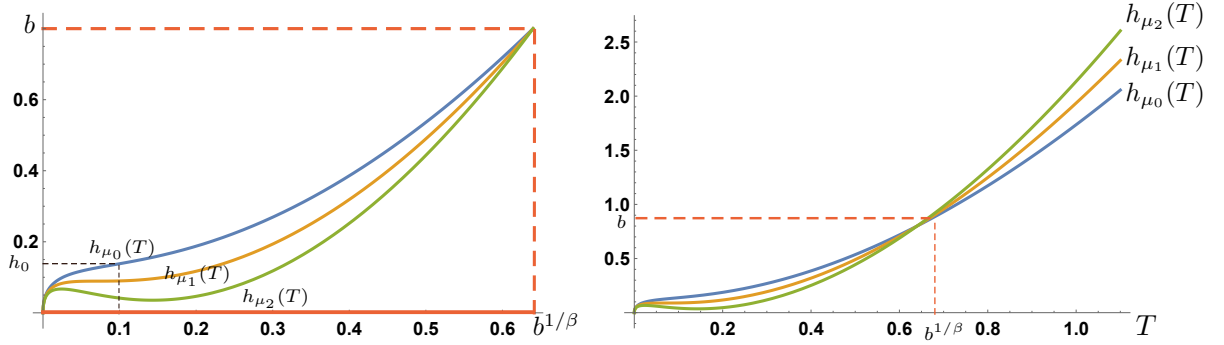

Combining the previous results it is deduced that $\displaystyle{h_{\mu}^{\prime}(T)}$ have exactly two zeros, which shall be called  $T_{L}=T_{L}(\mu)$ and $T_{R}=T_{R}(\mu)$ for all $\mu_{c}< \mu$. Furthermore, \footnote{The inequalities \eqref{inequalities for h-derivada} follow as result of all mentioned properties of $h_{\mu}^{\prime}(T)$ and $h_{\mu}^{\prime \prime}(T)$.} 
\begin{equation*}\label{inequalities for h-derivada}
0<T_{L}<T_{\ast}<T_{R}, \quad \text{and} \quad h_{\mu}^{\prime \prime}(T_{L})<0<h_{\mu}^{\prime \prime}(T_R).
\end{equation*}

Now, it is important to highlight the behavior of the zeros of $h_{\mu}^{\prime}(T)$ for $\mu>\mu_{c}$. Firstly, notice that
\[
h_{\mu}^{\prime}(T_{R})=0 \quad \Leftrightarrow \quad \mu\big((1+\beta)T_{R}^{\beta}-b\big)+\beta T_{R}^{\beta-1}=0 \quad \Rightarrow \quad T^{\beta}_{R}<\frac{b}{1+\beta}<b, 
\]
implying that $\displaystyle{0<T_{L}<T_{\star}<T_{R}<b^{1/\beta}}.$ 

Next, define the function 
\begin{align*}
    H_{L}: [\mu_{c},\infty[ &\to \mathbb{R}\\
    \mu &\to h_{\mu}(T_{L}(\mu)),
\end{align*}

which provides the value of $h_{\mu}(T)$ over the critical point $T_{L}$ for all $\mu_{c}\leq \mu$. A direct computation proves that
\[
T_{L}(\mu_{c})=T_{\star}(\mu_{c})=\Big(b\Big(\frac{1-\beta}{1+\beta}\Big)\Big)^{1/\beta}, \quad \text{and} \quad T^{\beta}_{\star}(\mu_c)-\mu_{c}b\, T_{\star}(\mu_{c})=0.
\]
From here, we deduce
\begin{align*}
H(\mu_{c})&=\mu_{c} T^{1+\beta}_{\star}(\mu_c)=b\Big(\frac{1-\beta}{1+\beta}\Big)^{1+\beta}<b,
\quad \text{and} \quad H^{\prime}(\mu)=\big(T^{\beta}_{L}(\mu)-b\big)T_{L}(\mu)<0.
\end{align*}
In consequence, $H_{L}(\mu)<b$ for all $\mu_{c}<\mu.$ Likewise, denoting $H_{R}(\mu):=h_{\mu}(T_{R}(\mu))$, by the monotonicity properties of $h_{\mu}^{\prime \prime}(T)$, it follows that $H_{R}(\mu)<b$, for all $\mu_{c}<\mu$.

\vspace{0.3 cm}
To sum up, for all $b,\beta \in [0,1[$ fixed, the function $\displaystyle{h_{\mu}(T)=\mu(T^{\beta}-b)T+T^{\beta}}$ satisfies:
\begin{itemize}
    \item[$\triangleright$] $h^{\prime}_{\mu}(T)$ has exactly two zeros $T_{L}(\mu)$, $T_{R}(\mu)$ such that 
    \[
    0<T_{L}(\mu)<T_{\star}(\mu)=\frac{1-\beta}{\mu(1+\beta)}<T_{R}(\mu)<b^{1/\beta}.
    \]
    \item [$\triangleright$] $h_{\mu}^{\prime \prime}(T)$ has exactly a simple zero at $T_{\star}(\mu)$ and
    \[
    h_{\mu}^{\prime \prime}(T_{L})<0<h_{\mu}^{\prime \prime}(T_R), 
    \]
    \item[$\triangleright$] If $\displaystyle{H_{R}(\mu):=h_{\mu}(T_{R}(\mu))}$ and $\displaystyle{H_{L}(\mu):=h_{\mu}(T_{L}(\mu))}$, then $\displaystyle{H_R(\mu)<H_{L}(\mu)}$,
\end{itemize}
for all $\displaystyle{\mu>\Big(\frac{1}{b}\Big(\frac{1-\beta}{1+\beta}\Big)^{\beta-1}\Big)^{\frac{1}{\beta}}}$.

\noindent
\subsection*{Bifurcation point.} The next lines are devoted to show the existence of $\hat{\mu}>\mu_{c}$ and $\hat{T}>0$ such that $h_{\mu}(T)$ has a saddle-node bifurcation point at $(\hat{T},\hat{\mu})$. Moreover,
\[
T_{R}(\mu) \to \hat{T}=[(1-\beta)b]^{1/\beta} \quad \text{if} \quad \mu \to \hat{\mu}:=\frac{1}{\beta[b(1-\beta)^{1-\beta}]^{1/\beta}}.
\]
\begin{prop} \label{Prop function h for mu positive 2} Let us consider $b,\beta \in ]0,1[,$ 
Then, the following properties hold
\begin{itemize}
    \item[1.] There exists a unique $\hat{T}\in \, ]0,b^{1/\beta}[$ and $\hat{\mu}>\mu_{c}$ such that 
\begin{equation}\label{bif system 2 new}
\begin{split}
h_{\hat{\mu}}(\hat{T})=0,\\
h^{\prime}_{\hat{\mu}}(\hat{T})=0.
\end{split}
\end{equation}
Moreover, $\hat{T}$ and $\hat{\mu}$ are given by
\[
\hat{T}=T_{bif}:=\big[(1-\beta)b\big]^{1/\beta} \quad \text{and} \quad \hat{\mu}=\mu_{bif}:=\frac{1}{\beta \big[b(1-\beta)^{1-\beta}\big]^{1/\beta}}.
\]
\item[2.] The equation $h_{\mu}(T)=0$ has a fold bifurcation at $\displaystyle{(T_{bif},\mu_{bif})}$.
\end{itemize}
\end{prop}

\begin{proof}

Let us consider the system
\begin{equation} \label{bif system on h}
\begin{split}
h_{\hat{\mu}}(\hat{T})=0,\\
h^{\prime}_{\hat{\mu}}(\hat{T})=0.
\end{split} \quad \Leftrightarrow \quad \begin{split}
T^{\beta}\big(1+\mu T\big)=\mu b T ,\\
T^{\beta-1}(\beta+\mu(1+\beta)T)=\mu b.
\end{split}
\end{equation}
As we know, the existence of zeros of $h^{\prime}_{\mu}(T)$ is only possible if $\mu\geq\mu_{c}$. Moreover, from expressions \eqref{bif system on h} we have
\[
T^{\beta}\big(1+\mu T\big)=T^{\beta}(\beta+\mu(1+\beta)T)\quad \Leftrightarrow \quad  T^{\beta}\big(1-\beta-\mu \beta T\big)=0 \quad \Leftrightarrow \quad  \hat{T}(\mu)=\frac{1-\beta}{\mu \beta}.
\]
By substituting this (unique) value  in the first equation in \eqref{bif system on h} it holds
\[
[\hat{T}(\mu)]^{\beta}=(1-\beta)b \quad \Rightarrow \quad \hat{T}(\mu)=\hat{T}:=((1-\beta)b)^{1/\beta}.
\]
In consequence
\[
\hat{\mu}=\frac{1-\beta}{\beta \hat{T}}=\frac{1}{\beta \big[b(1-\beta)^{1-\beta}\big]^{1/\beta}}. 
\]
A direct computation shows that $\mu_{c}<\hat{\mu}$. This proves the first part of the statement.

2. Recall that a parametric equation $f(x,\mu)=0$, with $f:I\times W \subset \mathbb{\R}^{2}\to \mathbb{R}$ $f\in C^{2}(I\times W)$ has a fold bifurcation at the point $(x_{0},\mu_{0})$ if 
\[
f(x_{0},\mu_{0})=0, \quad \partial_{x}f(x_{0},\mu_{0})=0, \quad \partial^{2}_{x}f(x_{0},\mu_{0})\neq 0, \quad \text{and} \quad \partial_{\mu}f(x_{0},\mu_{0})\neq 0.
\]
In such a case, there exist two intervals $I_{-}:=]\mu_{-},\mu_{0}[$, $I_{+}:=]\mu_{0},\mu_{+}[$ and $\varepsilon>0$ such that
\begin{itemize}
    \item[a.] If $\partial^{2}_{x}f(x_{0},\mu_{0})>0,$ and $\partial_{\mu}f(x_{0},\mu_{0})< 0$, then:
    \begin{itemize}
    \item[$\dagger$)] $f(x,\mu)$ has no zeros if $x \in \, ]x_{0}-\varepsilon,x_{0}+\varepsilon[$ and $\mu \in I_{-}$.
        \item[$\dagger$)] $f(x,\mu)$ has exactly two zeros if $x \in \, ]x_{0}-\varepsilon,x_{0}+\varepsilon[$ and $\mu \in I_{+}$.
    \end{itemize}
  \vspace{0.5 cm}  
\item[b.] If $\partial^{2}_{x}f(x_{0},\mu_{0})$ or $\partial_{\mu}f(x_{0},\mu_{0}),$ have reversed sign, then the previous conclusion is similar.
\end{itemize}

For the parametric equation 
\[
h_{\mu}(T)=0, \quad \Leftrightarrow \quad h_{\mu}=h(T,\mu)=0,
\]
with  $h:]0,b^{1/\beta}[\times ]0,\infty[ \to \mathbb{R}$, $(T,\mu)\to h(T,\mu)$ we have $\displaystyle{h(\hat{T},\hat{\mu})=0,}$ $\displaystyle{\partial_{T}h(\hat{T},\hat{\mu})=0}$ but also
\[
\partial^{2}_{T}h(\hat{T},\hat{\mu})=\frac{b}{\big[b(1-\beta)^{1-\beta}\big]^{2/\beta}}>0, \quad \text{and} \quad \partial_{\mu}h_{\mu}(\hat{T},\hat{\mu})=-\beta b \hat{T}<0.
\]
In consequence, there exists two intervals $I_{-}:=]\mu_{-},\hat{\mu}[$, $I_{+}:=]\hat{\mu},\mu_{+}[$ and $\varepsilon>0$ such that
    \begin{itemize}
    \item[$\ddagger$)] $h(T,\mu)$ has no zeros if $T \in \, ]\hat{T}-\varepsilon,\hat{T}+\varepsilon[$ and $\mu \in I_{-}$.
        \item[$\ddagger$)] $h(T,\mu)$ has exactly two zeros if $T \in \, ]\hat{T}-\varepsilon,\hat{T}+\varepsilon[$ and $\mu \in I_{+}$.
    \end{itemize}
\end{proof}
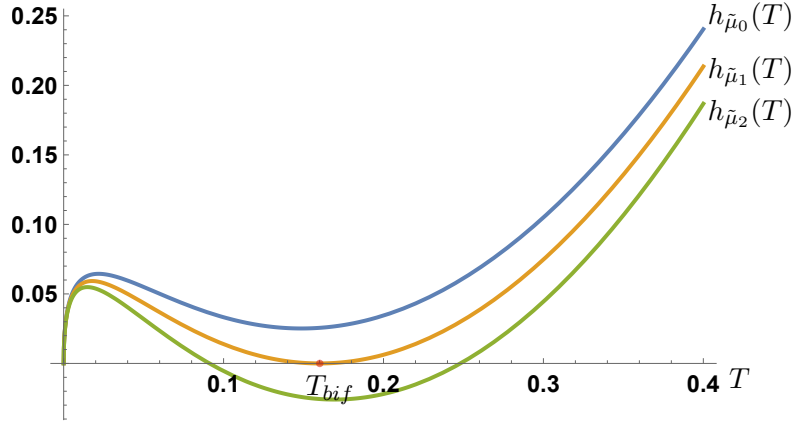
\begin{figure}[t]
\centering
\begin{overpic}[width = 0.6\textwidth, tics = 5]{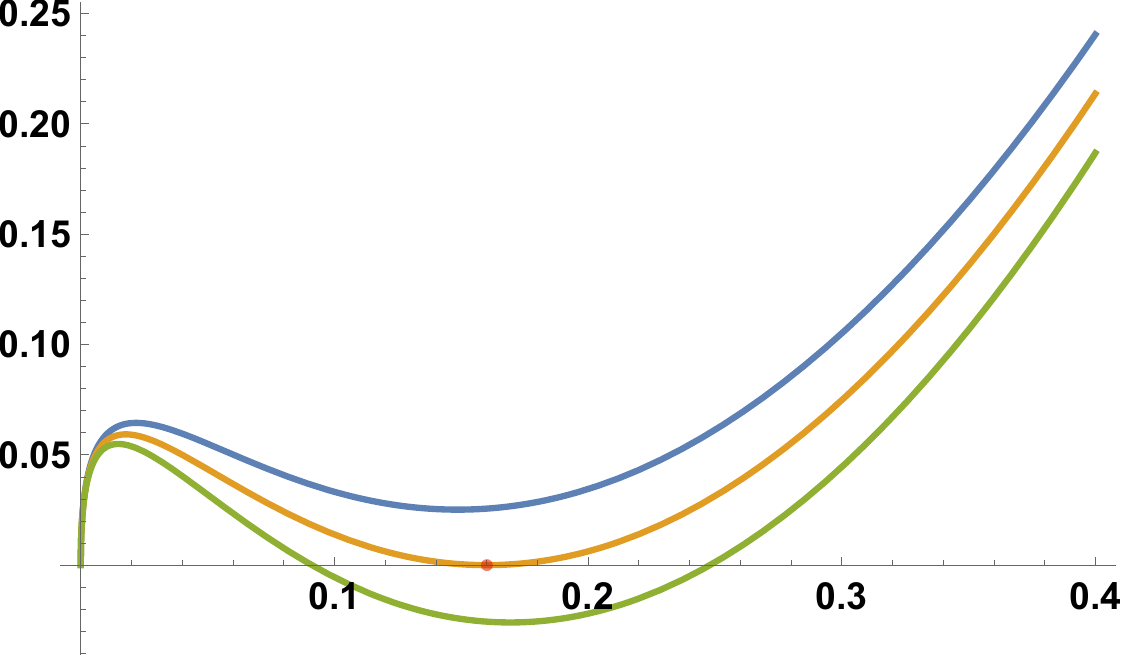
}
\put(98.4,56.1){$h_{\tilde{\mu}_{0}}(T)$}
\put(98.4, 48.5){$h_{\tilde{\mu}_{1}}(T)$}
\put(98.4, 42.5){$h_{\tilde{\mu}_{2}}(T)$}
\put(41.5, 3.5){$T_{bif}$}
\put(101.4, 4.5){$T$}
\end{overpic}
\caption{Graph of $h_{\mu}(T)$ with parameters $b=0.8$, and $\beta=0.5$. for $T\in \, ]0,5b^{1/\beta}/8]$. Three different values of $\tilde{\mu}$ are considered: $\tilde{\mu}_{n}=5.85+0.4 n$, $n \in \left\{0,1,2\right\}.$ Notice that $T_{bif}=4/25$ and $\tilde{\mu}_{1}= \mu_{bif}=25/4$. There are no solutions of  $h_{\tilde{\mu}_{0}}(T)=0$ and exactly two solutions of $h_{\tilde{\mu}_{2}}(T)=0$. }
\label{fig: bifurcation graph}
\end{figure}
We end this appendix by pointing out this new result about the solutions of $(\blacktriangle)$
\begin{remark}
By the monotonicity properties of $h_{\mu}^{\prime \prime}(T)$, it is known that $T_{R}(\mu)$ is the only critical point of $h_{\mu}(T)$ for $\mu>\mu_{c}$ such that 
$h_{\mu}^{\prime \prime}(T_{R})>0$.  Therefore, by the $C^{2}$-continuity of $h_{\mu}$ with respect to the parameter $\mu$ we have
\[
T_{R}(\mu)\to T_{bif} \quad \text{as} \quad \mu \to \mu_{bif},
\]
implying that $\displaystyle{H_{R}(\mu_{bif})=0.}$
\end{remark}
\begin{prop}
Let us consider the equation $(\blacktriangle)$. Assume that
\[
\mu_{c}:=\Big(\frac{1}{b}\Big(\frac{1-\beta}{1+\beta}\Big)^{\beta-1}\Big)^{\frac{1}{\beta}}<\mu \leq \mu_{bif}=\frac{1}{\beta \big[b(1-\beta)^{1-\beta}\big]^{1/\beta}}.
\]
Define
\[
H_L:=h_{\mu}(T_{L}(\mu)) \quad \text{and} \quad H_R:=h_{\mu}(T_{R}(\mu)),
\]
where $\displaystyle{T_{L,R}}(\mu)$ are only the  two critical points of $h_{\mu}(T)$.
\begin{itemize}
    \item[I)] If $ h_{0}=H_R$, then there are exactly two solutions of $(\blacktriangle)$ for all $T \in \, [0,b^{1/\beta}]$.
    \item[II)] If $H_{R}<h_{0}< H_{L}$, then there are exactly three solutions of $(\blacktriangle)$ for all $T \in \, [0,b^{1/\beta}]$.
    \item[III)]  If $H_{L}= h_{0}$, there are exactly two solutions of $(\blacktriangle)$ for all $T \in \, [0,b^{1/\beta}]$.
    \item[IV)]  If $0\leq h_{0}< H_R$ or $H_{L}< h_{0}$, there is exactly one solution of $(\blacktriangle)$ for all $T \in \, [0,b^{1/\beta}]$.
\end{itemize}
Finally, there exist $\varepsilon_{1,2}>0$ and an interval $\displaystyle{I_{+}:=]\mu_{bif},\mu_{bif}+\varepsilon_{1}[}$ such that, if $-\varepsilon_{2}<h_{0}\leq 0$ the equation $(\blacktriangle)$ admits exactly two solutions for all $T \in \, [0,b^{1/\beta}]$.
\end{prop}
\begin{remark}
In the current bibliography about the dynamics in a tumor immune system, the most usual growth function for the tumor cells is the Logistic growth function. In such a case
\[
h_{\mu}(T,\mu)=\mu T^{2}+(1-\mu b)T.
\]
Moreover, 
\[
\lim_{\beta\to 1}\mu_{c}=\frac{1}{b}, \quad \qquad \lim_{\beta\to 1}\mu_{bif}=\frac{1}{b}, \quad \text{\and} \quad 
0=T_{L}(1/b)=T_{\ast}\big(1/b \big)=T_{R}(1/b).
\]
This means that there are fewer equilibrium points for the logistic growth function case. In fact:
\end{remark}



\section*{Appendix B}
This appendix presents some results on the linear delay system
\[
(\underline{\dagger}) \qquad \dot{Y}(t)=AY(t)+BY(t-\tau_1)+CY(t-\tau_2),
\]
with  $A,B, C\in \mathbb{M}_{2\times 2}$ constant real matrices, $Y\in \mathbb{R}^{2}$ and $\tau_i \geq 0$. In the following:
\begin{itemize}
    \item[$\triangleright$] $\det A$ denotes the determinant of $A$.
    \item[$\triangleright$] $tr A$ denotes the trace of $A$. 
    \item[$\triangleright$] $(a^1|b^{2})$ is the matrix with the first column from $A$ and the second column from $B$.
\end{itemize}
Define the quadratic polynomial $P(\lambda,z,w)$ in $(\lambda,z,w):$
\[
\mathcal{P}(\lambda,z,w):=\lambda^2-(tr A) \lambda+\det A+(\det B)z^{2}+(\hat{c}-(tr B)\lambda)z+(\det C)w^2-\hat{d}(\lambda,z) w,
\]
where
\[
\hat{c}=\det (a^1|b^2)+\det (b^1|a^2) \quad \text{and} \quad 
\hat{d}(\lambda,z)=\det (c^1|m^2(\lambda,z))+\det (m^1(\lambda,z)|c^2), 
\]
with $\displaystyle{M(\lambda,z)=\lambda I_2-A-z B}$.
\begin{lemma}
Let $\lambda \in \C$ and $Y_{\ast}\neq 0_{\mathbb{R}^2}$. Then $\displaystyle{Y(t)=e^{\lambda t}Y_{\ast}}$ is a solution of $\displaystyle{(\underline{\dagger})}$ if $\lambda$ satisfies the associated characteristic equation
\[
\mathcal{P}(\lambda,e^{-\lambda \tau_1},e^{-\lambda \tau_2})=0.
\]
In particular, if $\tau=\tau_1=\tau_2$ the characteristic equation $\displaystyle{\mathcal{P}(\lambda,e^{-\lambda \tau},e^{-\lambda \tau})=0}$ is given by
\[
\lambda^2-(tr A) \lambda+\det A+(\det (B+C))e^{-2\lambda \tau}+(\tilde{c}-(tr( B+C)\lambda)e^{-\lambda \tau}=0,
\]
where, $\displaystyle{\tilde{c}=\det (a^1|(b+c)^2)+\det ((b+c)^1|a^2) }$.
\end{lemma}
\begin{proof}
We seek exponentially growing solutions of $\displaystyle{(\underline{\dagger})}$ of the form: $\displaystyle{Y(t)=e^{\lambda t}Y_{\ast}}$, with $Y_{\ast}\neq 0_{\mathbb{R}^2}$ and $\lambda \in \C$. Substituting into the equation $\displaystyle{(\underline{\dagger})}$, $\lambda$ satisfies
\[
e^{\lambda t}\big[\lambda I_2-A-e^{-\lambda \tau_1} B-e^{-\lambda \tau_2}C\big]Y_{\ast}=0 \quad \Rightarrow \quad 
\det\big(M-N\big)=0,
\]
with
\[
M=M(\lambda,z):=\lambda I_2-A-z B, \quad \text{and} \quad N:=e^{-\lambda \tau_2}C.
\]
From this, straightforward computations show that
\[
\det (M-N )=\det(M)+\det(N)-\det (n^1|m^2)-\det (m^1|n^2)
\]
finished the proof.
\end{proof}
\subsection*{Stability results and bifurcation of periodic solution for equal time-delay} Assume that for system $\displaystyle{(\underline{\dagger})}$ we have $\tau_1=\tau_2=\tau$, i.e., we have the delayed linear system
\[
(\underline{\ddagger}) \quad \dot{Y}(t)=AY(t)+(B+C)Y(t-\tau),
\]
with  $A,B, C\in \mathbb{M}_{2\times 2}$ constant real matrices, $Y=Y(t)\in \mathbb{R}^{2}$ and $\tau\geq 0$. Additionally, suppose that the corresponding characteristic equation is of the form
\[
\hat{\mathcal{P}}(\lambda,\tau)=\lambda^{2}+(\lambda_1+\lambda_2)\lambda+\lambda_1\lambda_2-Ne^{\tau \lambda},
\]
with $\lambda_i \in \mathbb{R}$ and $N\in \mathbb{R}^{-}$. In \cite{FREEDMAN} the authors present the following theorems that provide sufficient conditions under which the trivial solution $Y(t)=0_{\mathbb{R}^2}$ will be asymptotically stable, and the appearance of stable periodic limit cycles emerging in specific values of $\tau.$ 
\begin{miTeorema}[\textbf{A}]
Assume that $\displaystyle{\lambda_{1}+\lambda_{2}<0}$ and $\displaystyle{N<\lambda_1\lambda_{2}<-N.}$ Then $Y(t)=0_{\mathbb{R}^2}$ is asymp\-totically stable for all values of $\tau$ satisfying $\displaystyle{0\leq \tau<\frac{\lambda_{1}+\lambda_2}{N}}.$
\end{miTeorema}

\begin{miTeorema}[\textbf{B}] Assume that $\displaystyle{\lambda_{1}\lambda_{2}>-N}$. Then $Y(t)=0_{\mathbb{R}^2}$ is asymptotically stable for all $\tau \geq 0$.
\end{miTeorema}

\begin{miTeorema}[\textbf{C}](A Hopf bifurcation theorem)
Assume that 
\[
 \lambda_{1}+\lambda_{2}<0 \quad \text{and}\quad  N<\lambda_{1}\lambda_2<-N
\]
Then, there exists $\tau_{c}>0$ given as the smallest value of $\tau$ for which
\[
\hat{\mathcal{P}}(\lambda(\tau),\tau)=0,
\]
admits a solution $\lambda(\tau)=iy(\tau)$. Moreover, for $\tau<\tau_c$, $Y(t)=0_{\mathbb{R}^2}$ is asymptotically stable, while if $\tau>\tau_c$, $Y(t)=0_{\mathbb{R}^2}$ is unstable. Furthermore, as Furthermore, as  $\tau$  increases through $\tau_{c}$, $Y_{2}(t)=0$ bifurcates into ``small-amplitude'' periodic solutions, which are stable, i.e., $Y_{2}(t)=0$ loses its stability and undergoes a Hopf bifurcation.
\end{miTeorema}
\begin{remark}
 Equation $\hat{\mathcal{P}}(x+iy,\tau)=0$ can be divided into real and imaginary parts as follows
\begin{equation*}
\begin{split}
(x-\lambda_1)(x-\lambda_2)-y^2-N\cos(\tau y)e^{-\tau x}&=0,\\
2xy-(\lambda_1+\lambda_2)y+N\sin(\tau y)e^{-\tau x}&=0.
\end{split}
\end{equation*}
\end{remark}

\bibliographystyle{elsarticle-num}
\bibliography{bibliotumor}
\end{document}